\documentclass[reqno,a4paper, 11pt]{amsart}

\usepackage[a4paper=true,pdfpagelabels]{hyperref}
\usepackage{graphicx}

\usepackage[ansinew]{inputenc}
\usepackage{amsfonts,epsfig}
\usepackage{latexsym}
\usepackage{amsmath}
\usepackage{amssymb}

\newtheorem{theorem}{Theorem}
\newtheorem{lemma}[theorem]{Lemma}
\newtheorem{corollary}[theorem]{Corollary}
\newtheorem{proposition}[theorem]{Proposition}

\newtheorem{lettertheorem}{Theorem}
\newtheorem{letterlemma}[lettertheorem]{Lemma}

\theoremstyle{definition}

\theoremstyle{remark}

\numberwithin{equation}{section}

\newcommand{\dist}{{\rm dist}}
\newcommand{\essinf}{{\rm ess\,inf}}

\newcommand{\D}{\mathbb{D}}
\newcommand{\DD}{\widehat{\mathcal{D}}}
\newcommand{\N}{\mathbb{N}}
\newcommand{\M}{\mathcal{M}}
\newcommand{\R}{\mathbb{R}}

\newcommand{\C}{\mathbb{C}}

\renewcommand{\phi}{\varphi}

\newcommand{\T}{\mathbb{T}}

\def\a{\alpha}       \def\b{\beta}        \def\g{\gamma}
\def\d{\delta}           
     \def\om{\omega}      
\def\s{\sigma}       \def\t{\theta}       
                  \def\z{\zeta}
                  \def\vp{\varphi}

\def\LL{\mathcal{L}}

\def\dm{\,d(\omega\otimes m)}

\def\R{{\mathcal R}}

\newenvironment{Prf}{\noindent{\emph{Proof of}}}
{\hfill$\Box$ }

\begin{document}
\title[Bergman projection induced by kernel with integral representation]{Bergman projection induced by kernel with integral representation}


\thanks{
}

\author{Jos\'e \'Angel Pel\'aez}
\address{Jos\'e \'Angel Pel\'aez, Departamento de An\'alisis Matem\'atico, Universidad de M\'alaga, Campus de
Teatinos, 29071 M\'alaga, Spain} \email{japelaez@uma.es}

\author{Jouni R\"atty\"a}
\address{Jouni R\"atty\"a, University of Eastern Finland, P.O.Box 111, 80101 Joensuu, Finland}
\email{jouni.rattya@uef.fi}
\date{\today}

\author{Brett D. Wick}
\address{Brett D. Wick, Department of Mathematics, Washington University -- St. Louis, One Brookings Drive, St. Louis, MO USA 63130}
\email{wick@math.wustl.edu}

\thanks{This research was supported in part by the Ram\'on y Cajal program
of MICINN (Spain); by Ministerio de Econom\'{\i}a y Competitivivad,
Spain, project MTM2014-52865-P;  by   La Junta de Andaluc{\'i}a,
(FQM210) and (P09-FQM-4468);  by Academy of Finland project no.~268009, and by Faculty of Science and Forestry of University of Eastern Finland project no.~930349; by National Science Foundation Grants DMS \# 1560955 and \# 1603246.}
\date{\today}

\maketitle

\begin{abstract}
Bounded Bergman projections $P_\omega:L^p_\omega(v)\to
L^p_\omega(v)$, induced by reproducing kernels admitting the
representation
    $$
    \frac{1}{(1-\overline{z}\zeta)^\gamma}\int_0^1\frac{d\nu(r)}{1-r\overline{z}\zeta},
    $$
and the corresponding (1,1)-inequality are characterized in terms of
Bekoll\'e-Bonami-type conditions. The two-weight inequality for the
maximal Bergman projection $P^+_\omega:L^p_\omega(u)\to
L^p_\omega(v)$ in terms of Sawyer-testing conditions is also
discussed.
\end{abstract}

\section{Introduction and main results}

Let $\DD$ denote the set of positive Borel measures $\om$ on $[0,1)$
such that $\widehat{\om}(r)=\int_r^1d\om(r)\le
C\widehat{\om}\left(\frac{1+r}{2}\right)$ for some $C=C(\om)>0$. For
$0<p<\infty$ and $\omega\in\DD$, the weighted Bergman space
$A^p_\omega$ consists of analytic functions $f$ in the unit disc
$\D=\{z\in\C:|z|<1\}$ such that
    $$
    \|f\|_{A^p_\omega}^p=\int_\D|f(z)|^p\dm(z)<\infty,
    $$
where $\dm(re^{i\t})=rd\om(r)d\t$. As usual, we write~$A^p_\alpha$
for the standard weighted Bergman space induced by the measure $\om$
for which $\dm(z)=(\a+1)(1-|z|^2)^\alpha dA(z)=dA_\a(z)$, where
$-1<\alpha<\infty$ and $dA(z)=\frac{dxdy}{\pi}$ denotes the
normalized Lebesgue area measure on $\D$. For simplicity, we also
write $\int_Ef(z)\dm(z)=(f\om)(E)$ for each non-negative $f$.

By the proof of \cite[Theorem~3.3]{PelSum14}, for $\om\in\DD$, the
norm convergence in $A^2_\om$ implies the uniform convergence on
compact subsets, and hence each point evaluation $L_z(f)=f(z)$ is a
bounded linear functional in the Hilbert space $A^2_\om$. Therefore
there exist unique reproducing kernels $B^\om_z\in A^2_\om$ with
$\|L_z\|=\|B^\om_z\|_{A^2_\om}$ such that
    $$
    f(z)=\langle f,B^\om_z\rangle_{A^2_\om}=\int_\D f(\z)\overline{B^\om_z(\z)}\dm(\z),\quad f\in A^2_\om.
    $$
The Bergman projection
    \begin{equation*}\label{intoper}
    P_\om(f)(z)=\int_{\D}f(\z)\overline{B^\om_{z}(\z)}\dm(\z)
    \end{equation*}
is an orthogonal projection from $L^2_\om$ to $A^2_\om$ and it is
closely related to the maximal Bergman projection
    \begin{equation*}\label{intoper+}
    P^+_\om(f)(z)=\int_{\D}f(\z)\left|B^\om_{z}(\z)\right|\dm(\z).
    \end{equation*}

For a positive Borel measure $\om$ on~$[0,1)$, a positive
$(\omega\otimes m)$-integrable function $v$ is called an
$\om$-weight. If $\omega\otimes m$ is the normalized Lebesgue area
measure, then an $\om$-weight is simply called a weight. For
$0<p<\infty$ and an $\om$-weight $v$, the Lebesgue space
$L_\om^{p}(v)$ consists of $f$ such that
    $$
    \|f\|^p_{L_\om^{p}(v)}=\int_{\D}|f(z)|^p v(z)\dm(z)<\infty.
    $$

The boundedness of projections on $L^p$-spaces is an intriguing
topic which presents obvious mathematical difficulties and has
plenty of  applications in operator
theory~\cite{Arrthesis,BB,B,CP2,DHZ,PelRatproj,PottRegueraJFA13,ZeyTams2012,Zhu}.
It is known that for $1<p<\infty$ and $\dm=dA_\a$,
    \begin{equation}\label{twoweight}
    \|P_\omega(f)\|_{L^p_\om(v)}\le C\|f\|_{L^p_\om(v)},\quad f\in L^p_\om(v),
    \end{equation}
if and only if $v$ satisfies the Bekoll\'e-Bonami condition
    \begin{equation}\label{eq:BB}
    \begin{split}
    &B_{p,\alpha}(v)=\sup_{S}
    \frac{ vA_\a(S) \left(v^{\frac{-p'}{p}}A_\a(S)\right)^{\frac{p}{p'}}}
    {\left(A_\a(S))\right)^p}<\infty,
    \end{split}
    \end{equation}
where the supremum is taken over all the Carleson squares $S$ in
$\D$~\cite{BB,B}. In the above result $P_\a$ can be replaced by the
maximal Bergman projection $P^+_\a$~\cite{BB,B}, and
$\|P^+_\a\|\lesssim
B_{p,\alpha}(v)^{\max\left\{1,\frac{1}{p-1}\right\}}$
by~\cite{PottRegueraJFA13}. It is also known~\cite{B,BekCan86,DHZ}
that the weak (1,1) inequality
    \begin{equation*}\label{jap:weak11}
    v A_\a\Bigl(\Bigl\{z\in\D: |P_\a(f)(z)|>\lambda\Bigr\}\Bigr)\le
    C\frac{\|f\|_{L^1_{\a}(v)}}{\lambda}
    \end{equation*}
is equivalent to
    $$
    M_{\alpha}(v)(z)\le C v(z),\quad {\rm a.e.~}z\in\D,
    $$
for the weighted maximal function
    $$
    M_{\alpha}(v)(z)=\sup_{z\in
    D(a,r)}\frac{vA_\a(D(a,r)\cap\D)}{A_\a(D(a,r)\cap\D)},\quad z\in\D.
    $$
Here $D(a,r)$ denotes the Euclidean disc of center $a$ and radius
$r$.

An immediate difficulty in controlling \eqref{twoweight} for a given
measure $\om\in\DD$ is the lack of an explicit expression for the
Bergman kernel $B^\om_z$. Writing $\om_x=\int_0^1r^x\om(r)\,dr$, the
normalized monomials $z^n/\sqrt{2\om_{2n+1}}$
form the standard orthonormal basis of $A^2_\om$, and hence
    \begin{equation}\label{repker}
    B^\om_z(\z)=\sum_{n=0}^\infty\frac{(\z\overline{z})^n}{2\om_{2n+1}},\quad z,\z\in\D.
    \end{equation}
This formula and a decomposition norm theorem was recently used to
obtain a precise estimate for the $L^p_v$-integral of~$B^\om_{z}$
when $v,\om$ are weights in $\DD$~\cite[Theorem~1]{PelRatproj}. With
the aid of these estimates, \eqref{twoweight} was characterized in
the case when $\om$ and $v$ are weights in the class
$\R$~\cite{PelRatproj,PelRatproj2}. A positive Borel measure $\om$
on $[0,1)$ belongs to $\R$, if there exist $C=C(\om)>0$,
$\g=\g(\om)>0$ and $\b=\b(\om)\ge\gamma$ such that
    \begin{equation}
    \begin{split}\label{eq:regular}
    C^{-1}\left(\frac{1-r}{1-t}\right)^{\g}\widehat{\om}(t)\le\widehat{\om}(r)\le C\left(\frac{1-r}{1-t}\right)^{\b}\widehat{\om}(t),\quad 0\le r\le
    t<1.
    \end{split}
    \end{equation}

In view of the above results two immediate questions arise. First,
is it possible to extend the classical Bekoll\'e-Bonami's results to
projections $P_\om$ induced by measures in $\DD$? Second, is it
possible to consider other weights than just those in $\R$ in the
same spirit as in \cite[Theorem~3]{PelRatproj}?

A natural approach to these question is to employ tools from
harmonic analysis. However, it seems that to do so one needs the
Bergman kernel $B^\om_z$ to have some structure. The first result of
this study shows that certain doubling measures induce kernels with
suitable properties for our purposes.

\begin{theorem}\label{th:kernelintegral}
Let $\nu$ be a finite positive measure supported on $[0,1]$ such
that $\int_0^1\frac {d\nu(r)}{1-r}$ diverges. Then there exists
$\om\in\DD$ such that
    $$
    B^\om_z(\z)=\frac1{1-\overline{z}\z}\int_0^1\frac{d\nu(r)}{1-r\overline{z}\z},\quad z,\z\in\D.
    $$
\end{theorem}

Since each kernel $B^\om_z$ induced by $\om\in\DD$ has the
representation \eqref{repker}, and
    $$
    \frac{1}{1-z}\int_0^1\frac{d\nu(r)}{1-rz}=\sum_{n=0}^\infty\left(\int_0^1\frac{1-r^{n+1}}{1-r}\,d\nu(r)\right)z^n,
    $$
the proof of Theorem~\ref{th:kernelintegral} basically boils down to
solving a Hausdorff moment problem. In Section~\ref{Sec2} we will
prove a more general result from which
Theorem~\ref{th:kernelintegral} immediately follows.

Next we focus on extending the classical Bekolle-Bonami's results
for those measures $\om\in\R$ that induce kernels admitting the
representation
     $$
     B^\om_z(\z)=\frac{1}{(1-\overline{z}\z)^\gamma}\int_0^1\frac{d\nu(r)}{1-r\overline{z}\z},\quad
     z,\z\in\D,
     $$
for some $\gamma\ge 1$. For $1<p<\infty$ and $\om\in\DD$, an
$\om$-weight $v$ belongs $B_{p,\omega}$ if
    $$
    B_{p,\omega}(v)=\sup_{S}\frac{(v\om)(S)}{\om(S)}\left(\frac{\left(v^{-\frac{p'}{p}}\om\right)(S)}{\om(S)}\right)^{\frac{p}{p'}}<\infty.
    $$

\begin{theorem}\label{Theorem:BB-generalized}
Let $1<p<\infty$ and $\om\in\R$ such that $B^\om_z$ admits the
representation
    \begin{equation}\label{Eq:kernel-representation}
    B^\om_z(\z)=\frac{1}{(1-\overline{z}\z)^\gamma}\int_0^1\frac{d\nu(r)}{1-r\overline{z}\z},\quad z,\z\in\D,
    \end{equation}
for some $\gamma\ge1$ and a positive measure $\nu$ supported on
$[0,1]$. For an $\om$-weight $v$, the following statements are
equivalent:
    \begin{enumerate}
    \item[\rm(i)] $P^+_\om:L^p_{\omega}(v)\to L^p_{\omega}(v)$ is bounded;
    \item[\rm(ii)] $P_\om:L^p_{\omega}(v)\to L^p_{\omega}(v)$ is bounded;
    \item[\rm(iii)] $P_\om: L^p_{\omega}(v)\to L^{p,\infty}_{\omega}(v)$ is bounded;
    \item[\rm(iv)] $v\in B_{p,\om}$.
    \end{enumerate}
Moreover,
    $$
    \left\Vert P^+_{\om}\right\Vert_{L^p_{\omega}(v)\to L^p_{\omega}(v)}\lesssim
    B_{p,\omega}(v)^{\max\left\{1,\frac{1}{p-1}\right\}}.
    $$
\end{theorem}

To prove (iii)$\Rightarrow$(iv) in
Theorem~\ref{Theorem:BB-generalized}, we estimate
$|B^\om_{z_0}(\z)-B^\om_{z}(\z)|$ upwards for suitable chosen
$z,z_0,\z$, and we also establish the useful relation
    $$
    \int_0^1\frac{d\nu(r)}{1-rx}\asymp\frac{(1-x)^{\gamma-1}}{\widehat{\om}(x)},\quad x\in[0,1),
    $$
for the measures $\nu$ and $\om$. The proof of (iv)$\Rightarrow$(i)
is based on known ideas of controlling $P^+_\om$ by two discrete
dyadic operators \cite{GarJones82,Mei03,PottRegueraJFA13}, and it is
done in the case of a more general operator.
Theorem~\ref{Theorem:BB-generalized} is proved in
Section~\ref{Sec3}.

Now we turn to study of the weak $(1,1)$-inequality. For a positive
Borel measure $\om$ on $[0,1)$, the weighted maximal function of
$f\in L^1_\om$ is
    \[
    M_{\om}(f)(z)=\sup_{z\in D(a,r)}\frac{\int_{D(a,r)\cap\D}
    |f(\z)|\dm(\z)}{\om\Bigl(D(a,r)\cap\D \Bigr)},\quad z\in\D.
    \]
A non-negative function $v\in L^1_{\om,{\rm loc}}$ belongs to
$B_{1,\om}$ if there exists a constant $C=C(v,\om)>0$ such that
    $$
    M_\om(v)(z)=\sup_{a:z\in D(a,r)}\frac{\int_{D(a,r)\cap\D}
    v(\z)\dm(\z)}{\om(D(a,r)\cap\D)}\le Cv(z)
    $$
for almost every $z\in\D$. The infimum of such constants is denoted
by $B_{1,\om}(v)$. In order to obtain the weak $(1,1)$-inequality we
use the classical Calder\'on-Zygmund decomposition for functions in
$L^1_\om$. This causes the extra hypothesis on $\om$ appearing in
the statement of the following result, the proof of which is given
in Section~\ref{Sec4}.

\begin{theorem}\label{th:w11}
Let $\om\in\R$ be such that
$\om([a,b])\asymp\om\left(\left[a,\frac{a+b}2\right]\right)\asymp\om\left(\left[\frac{a+b}2,b\right]\right)$
for all $0\le a,b\le1$ and $B^\om_z$ admits the representation
\eqref{Eq:kernel-representation} for some $\gamma\ge1$ and a
positive measure $\nu$ supported on $[0,1]$. For a $\omega$-weight
$v$,  the following statements are equivalent:
\begin{enumerate}
\item[\rm(i)] $P^+_\om: L^1_{\omega}(v)\to L^{1,\infty}_{\omega}(v)$ is bounded;
\item[\rm(ii)] $P_\om: L^1_{\omega}(v)\to L^{1,\infty}_{\omega}(v)$ is bounded;
\item[\rm(iii)] $v\in B_{1,\om}$.
\end{enumerate}
\end{theorem}

In Theorems~\ref{Theorem:BB-generalized} and \ref{th:w11} one of the
essential hypothesis is $\om\in\R$ while
Theorem~\ref{th:kernelintegral} concerns measures in $\DD$. However,
if $\gamma$ appearing in \eqref{Eq:kernel-representation} is
strictly larger than one, then $\om\in\R$ by Lemma~\ref{le:shi1}
below. It is also worth noticing that kernels admitting the
representation \eqref{Eq:kernel-representation} with $\gamma=1$ and
their connection to logarithmically subharmonic weights have been
discussed earlier in \cite{ShimorinCV02}, and the starting point for
our consideration towards Theorem~\ref{th:kernelintegral} has
similarities with arguments used there.

The two-weight inequality $\|P^+(f)\|_{L^p_u}\le C\|f\|_{L^p_v}$ was
recently characterized in terms of testing conditions on the
indicators of Carleson squares~\cite{AlPoRe}. The last of our main
results offers a generalization of this result to the class of
radial weights with kernels of the form
\eqref{Eq:kernel-representation}. We write $1_{E}$ for the
characteristic function of the set $E$, and write $\M_h$ for the
multiplication operator $\M_h(f)=fh$.

\begin{theorem}\label{th:P+logsub}
Let $1<p<\infty$, and let $\om$ be a finite positive measure on
$[0,1]$ such that $B^\om_z$ admits the representation
\eqref{Eq:kernel-representation} for some $\gamma\ge1$ and a
positive measure $\nu$ supported on $[0,1]$. Let $v,u$ be
$\om$-weights and denote $\sigma=v^{1-p'}$. Then $P^{+}_{\om}:
L^p_\om(v)\to L^p_\om(u)$ is bounded if and only if there exist
constants $C_0=C_0(p,v,u,\omega)>0$ and
$C^\star_0=C^\star_0(p,v,u,\omega)>0$ such that
    \begin{equation}\label{j1lg}
    \left\|\M_{u^{1/p}}P^{+}_{\om}\M_{\sigma^{1/{p'}}}(1_{S}\sigma^{1/p})\right\|_{L^p_\om}\le C_0\left\|1_{S}\sigma^{1/p}\right\|_{L^p_\om}
    \end{equation}
and
    \begin{equation}\label{j2lg}
    \left\|\M_{\sigma^{1/{p'}}}P^{+}_{\om}\M_{u^{1/p}}(1_{S}u^{1/p'})\right\|_{L^{p'}_\om}\le C^\star_0\left\|1_{S}u^{1/p'}\right\|_{L^{p'}_\om}
    \end{equation}
for all Carleson squares $S\subset\D$. Moreover, there exists a
constant $C_1=C_1(p,\om)>0$ such that
    $$
    \left\|P^{+}_{\om}\right\|_{L^p_\om(v)\to L^p_{\om}(u)}\le C_1(C_0+C^\star_0).
    $$
\end{theorem}

Theorem~\ref{th:P+logsub} is deduced from a more general result in
Section~\ref{Sec5}.

\section{Integral formula for the Bergman kernel}\label{Sec2}

The solution of the Hausdorff moment problem says that for a given
sequence $\{m_n\}_{n=0}^\infty$ of positive numbers there exists a
positive Borel measure supported on $[0,1]$ such that
\begin{equation}\label{momentproblem}
m_n=\int_0^1 s^n\,d\mu(s),\quad n\in\N\cup\{0\},
\end{equation}
if and only if, the sequence is completely monotonic i.e. $(-1)^k
(\Delta^k m)_n\ge 0$, where $(\Delta m)_n=m_{n+1}-m_n$ is the
discrete difference operator and
    $$
    (\Delta^k m)_n=(\Delta\Delta^{k-1}m)_n=(\Delta^{k-1}m)_{n+1}-(\Delta^{k-1} \Delta m)_n,\quad k\in\N\setminus\{1\}.
    $$
A function $f$ is completely monotonic on $[0,\infty)$, if
    $$
    (-1)^{k}f^{(k)}(x)\ge 0,\quad x>0,\quad k\in\N\cup\{0\},
    $$
and $f:[0,\infty)\to[0,\infty)$ is Bernstein, if
    $$
    (-1)^{k}f^{(k)}(x)\le 0,\quad x>0,\quad k\in \N.
    $$
The first two of the following basic properties are easy to verify,
for the third and fourth ones, see~\cite[Theorem~3.7]{Schibook}
and~\cite[Theorem~1]{WidTAMS31}, respectively:
\begin{itemize}
\item[\rm(1)] If $f_1$ and $f_2$ are completely monotonic functions, so are $f_1+f_2$ and $f_1f_2$;
\item[\rm(2)] If $f_1$ and $f_2$ are  Bernstein functions, so is $f_1+f_2$;
\item[\rm(3)] If $f_1$ is completely monotonic and $f_2$ is a Bernstein function, then $f_1\circ f_2$ is a completely monotonic function;
\item[\rm(4)] If $f$ is completely monotonic, then $\{f(a+n)\}_{n=0}^\infty$ is a completely monotonic sequence for each
$a>0$.
\end{itemize}

Theorem~\ref{th:kernelintegral} follows from the following result.

\begin{theorem}\label{pr:general}
Let $F:[0,\infty)\to (0,\infty)$ be a $C^\infty$-function and
$\varphi(z)=\sum_{n=0}^\infty\widehat{\varphi}(n)z^n $ a non-trivial
analytic function such that $1/F$ is completely monotonic and
$F(a+2n+1)=\sum_{j=0}^{n}\widehat{\varphi}(j)$ for some
$a\in(0,\infty)$ and all $n\in\N\cup\{0\}$. Then there exists a
positive Borel measure $\om$ on $[0,1]$ such that
    \begin{equation}\label{repre3}
    \frac{\varphi(z)}{1-z}=\sum_{k=0}^\infty\frac{z^{k}}{2\om_{2k+1}},\quad z\in\D.
    \end{equation}
Moreover, if $\lim_{n\to \infty}F(a+2n+1)=\infty$,
$F\left(a+2n\right)\lesssim F\left(a+n\right)$ and there exists a
positive constant $M>1$ such that $\lim_{n\to
\infty}\frac{M^n}{F(a+n)}=\infty$, then $\om\in\DD$ and
    $$
    \frac{\varphi(\overline{z}\z)}{1-\overline{z}\z}=B_z^\om(\z),\quad \z,z\in\D.
    $$
\end{theorem}

\begin{proof}
Since $1/F$ is completely monotonic, there exists a positive Borel
measure $\om$ on $[0,1]$ such that
$F\left(a+m\right)=\frac{1}{2\om_m}$ for all $m\in\N\cup\{0\}$. In
particular, $F(a+2n-1)=\frac{1}{2\om_{2n-1}}$ for all $n\in\N$.
Therefore
    \begin{equation*}
    \begin{split}
    \frac{\varphi(z)}{1-z}
    &=\frac1z\left(\sum_{k=1}^\infty z^k\right)\left(\sum_{j=0}^\infty\widehat{\varphi}(j)z^j\right)
    =\frac1z\sum_{n=1}^\infty\left(\sum_{j=0}^{n-1}\widehat{\varphi}(j)\right)z^n\\
    &=\sum_{n=1}^\infty F(a+2n-1)z^{n-1}
    =\sum_{n=1}^\infty\frac{z^{n-1}}{2\om_{2n-1}}
    =\sum_{k=0}^\infty\frac{z^{k}}{2\om_{2k+1}},
    \end{split}
    \end{equation*}
and thus \eqref{repre3} is proved. Moreover,
    \begin{equation}
    \begin{split}\label{1n}
    \om(\{1\})\le\om\left(\left[1-\frac1{2n-1},1\right]\right)\lesssim\om_{2n-1}=\frac{1}{2F(a+2n-1)}\to0,\quad
    n\to\infty,
    \end{split}
    \end{equation}
and if $m=\frac{1}{M-1}+1$, then there exists a constant
$C=C(\om)>0$ such that
    \begin{equation}\label{m}
    \widehat{\om}(0)\le C\widehat{\om}\left(1-\frac{1}{m}\right).
    \end{equation}
For otherwise we would have
$\widehat{\om}\left(1-\frac{1}{m}\right)=0$, and then
    $$
    \frac{1}{2F(a+n)}=\om_{n}=\int_0^{1-\frac{1}{m}}r^{n}\,d\om(r)
    \le\left(1-\frac{1}{m}\right)^{n}\int_0^{1-\frac{1}{m}}\,d\om(r)=M^{-n}\int_0^{1-\frac{1}{m}}\,d\om(r),
    $$
which yields a contradiction with the hypothesis $\lim_{n\to
\infty}\frac{M^n}{F(a+n)}=\infty$. Since
$\om_n=\frac{1}{2F\left(a+n\right)}\lesssim\frac{1}{2F\left(a+2n\right)}=\om_{2n}$,
this together with \eqref{1n} and \eqref{m} implies $\om\in\DD$ by
\cite[Lemma~2.1]{PelSum14}.
\end{proof}

\begin{Prf}{\em{Theorem~\ref{th:kernelintegral}.}}
Consider the function
$\vp(z)=\int_0^1\frac{d\nu(r)}{1-rz}=\sum_{j=0}^\infty \nu_j z^j$,
and observe that $\sum_{j=0}^n
\nu_j=\int_0^1\frac{1-r^{n+1}}{1-r}d\nu(r)=F(2n+1+1/2)$ for
    $$
    F(x)=\int_{0}^1\frac{1-r^{\frac{x+\frac{1}{2}}{2}}}{1-r}\,d\nu(r),\quad 0\le x<\infty.
    $$
Since $f(x)=1/x$ is completely monotonic and $F$ is a Bernstein
function as is seen by direct calculations, $1/F$ is completely
monotonic. Therefore, by Theorem~\ref{pr:general}, there exists a
positive Borel measure $\om$ on $[0,1]$ such that
    \begin{equation*}
   \frac1{1-z}\int_0^1\frac{d\nu(r)}{1-rz}=\sum_{k=0}^\infty\frac{z^{k}}{2\om_{2k+1}},\quad z\in\D.
    \end{equation*}
Moreover, $\om$ is supported on $[0,1)$ because
$\sum_{j=0}^\infty\nu_j=\infty$, and it satisfies \eqref{m} because
$\lim_{n\to \infty} \frac{M^n}{\int_0^1
\frac{1-r^{n+1}}{1-r}d\nu(r)}\ge\lim_{n\to \infty}
\frac{M^n}{n\nu([0,1])}=\infty$. Since
$1-r^\frac{2m+1}{2}\le2(1-r^{\frac{m+1}{2}})$ for all
$m\in\N\cup\{0\}$, we also have $\om_m\lesssim\om_{2m}$. Hence
$\om\in\DD$ by \cite[Lemma~2.1]{PelSum14}, and
    $$
    B^\om_z(\z)=\frac1{1-\overline{z}\z}\int_0^1\frac{d\nu(r)}{1-r\overline{z}\z},\quad z,\z\in\D.
    $$
\end{Prf}

\medskip

Theorems~\ref{th:kernelintegral} and~\ref{pr:general} can be used to
provide examples of concrete Bergman reproducing kernels:
\begin{enumerate}
\item If $\nu$ is the Lebesgue measure, Theorem~\ref{th:kernelintegral} gives the kernel
    $$
    B^\om_z(\z)=\frac1{1-\overline{z}\z}\frac1{\overline{z}\z}\log\frac{1}{1-\overline{z}\z}.
    $$
\item Theorem~\ref{pr:general}  allows to recover the
well-known formula of the Bergman kernels induced by the standard
weights $\om(z)=(\a+1)(1-|z|^2)^\a$, $\a>-1$. Indeed, by choosing
$a=1$ and $F(x)=\frac{1}{\beta(x/2,\a+1)}$, have
$F(a+2j+1)=F(2j+2)=\frac{1}{\beta(j+1,\a+1)}$. It is clear that
$1/F$ is completely monotonic on $[0,\infty)$ and the function $\vp$
associated to $F$ is
    $$
    \frac{1}{(1-z)^{\a+1}}=
    \frac{1}{\beta(1,\a+1)}+\sum_{j=1}^\infty\left(\frac{1}{\beta(j+1,\a+1)}-\frac{1}{\beta(j,\a+1)}\right)z^j.
    $$

\item Let $\varphi(z)=\frac{\log\frac{e}{1-z}}{1-z}=1+\sum_{j=1}^\infty\left(1+\sum_{k=1}^{j}\frac{1}{k} \right)z^j$ so that $\sum_{j=0}^n\widehat{\varphi}(j)=1+(n+1)\int_{0}^1\frac{1-s^{n+1}}{1-s}\,ds$, and choose $a=\frac{1}{2}$ and $F(x)=1+\frac{x+1}{2}\int_{0}^1
\frac{1-s^{\frac{x+\frac{1}{2}}{2}}} {1-s}\,ds$ so that
$F(1/2+2n+1)=\sum_{j=0}^n\widehat{\varphi}(j)$. Since
$x\to\frac{x+1}{2}$ and $x\to \int_{0}^1
\frac{1-s^{\frac{x+\frac{1}{2}}{2}}} {1-s}\,ds$ are Bernstein
functions on $[0,\infty)$, $F$ is completely monotonic on
$[0,\infty)$. Moreover, it is clear that $F$ satisfies the rest of
hypothesis of Theorem~\ref{pr:general}, and hence there exists
$\om\in\DD$ such that
    $$
    B^\om_z(\z)=\log\frac{1}{(1-\overline{z}\z)^2}\frac{1}{(1-\overline{z}\z)^2}.
    $$
\end{enumerate}

\section{Generalization of the result of Bekoll\'e and Bonami}\label{Sec3}

For a positive Borel measure $\mu$ on $\D$ and an analytic function $\Psi$ in $D(1,1)$ such that its restriction to the interval $(0,2)$ is a real positive function, define
    \begin{equation}\label{projposi}
    P^+_{\Psi,\mu}(f)(z)=\int_{\D}\left|\frac{\Psi(1-\overline{\z}z)}{1-\overline{\z}z}\right| f(\z)\,d\mu(\z),\quad f\in L^1_\mu,\quad z\in\D.
    \end{equation}
To obtain a dyadic model for the operator $P^+_{\Psi,\mu}$ we define the dyadic grids
    \begin{equation}\label{concretegrid}
    \displaystyle \mathcal D^{\beta} = \left\{ I^\beta_{j,m}:
    \,j\in\N\cup\{0\},\,m\in\N\cup\{0\},\,0\le m\le 2^{j}-1\right\},\quad \beta\in \{0,1/2\},
    \end{equation}
where
    $$
    I^\beta_{j,m}=\left\{e^{i\theta}: \theta\in  \left[ \frac{4\pi\left( m+\beta\right)}{2^{j}},
    \frac{4\pi\left( m+1+\beta\right)}{2^j}\right) \right\}.
    $$
For each interval $I\subset\T$, with the convention $I=(\alpha,\b)=(\alpha+e^{i2\pi j},\b+e^{i2\pi j})$ for all $j\in \mathbb N\cup\{0\}$, there exists $K=K(I)\in\mathcal{D}\cup\mathcal{D}^{1/2}$ such that $I\subset K$ and $|K|\leq 4|I|$. Define the positive dyadic kernels
    \begin{equation}\label{e.dyadicp}
    K^{\beta}_\Psi(z,\z)=\sum_{I\in \mathcal D^{\beta}} \frac{1_{S(I)}(z)1_{S(I)}(\z)\Psi(|I|)}{|I|},\quad z,\z\in\D,\quad \beta\in \{0,1/2\},
    \end{equation}
where $S(I)=\{re^{i\theta}:\,\, 1-|I|\leq r<1,\, e^{i\theta}\in I\}$ is the Carleson square associated to $I$, and
$|I|$ stands for the normalized arc-length of the interval $I$. For this kernel and a positive Borel measure $\mu$ on $\D$, define the dyadic operator
    \begin{equation}\label{e.drm}
    P^{\beta}_{\Psi,\mu}(f)(z)=\sum_{I\in\mathcal D^{\beta}}\left\langle f, \frac{1_{S(I)} \Psi(|I|)}{|I|} \right \rangle_{L_\mu^2} 1_{S(I)}(z),
    \quad z\in\D,
    \end{equation}
and write
    \begin{equation}\label{e.kernel}
    K_{\Psi}(z,\z)= \frac{\Psi(|1-\overline{\z}z|)}{|1-\overline{\z}z|}, \quad z,\z\in\D,
    \end{equation}
for short.

The first lemma relates the operator $P^+_{\Psi,\mu}$ to the sum of the dyadic operators $P^{\beta}_{\Psi,\mu}$, $\beta\in \{0,1/2\}$, by means of a simple pointwise estimate for the inducing kernels.

\begin{lemma}\label{p.pointwise}
Let $\Psi$ be a positive essentially decreasing function on $(0,2)$
such that $\Psi(t)\le C \Psi(2t)$ for all $t\in(0,1)$ and for some
$C=C(\Psi)>0$. Then there exists a constant $C_1=C_1(\Psi)>1$ such
that
    \begin{equation}\label{eq:kernelcomp}
    C_1^{-1}\left(K_{\Psi}^{0}(z,\z)+K_{\Psi}^{1/2}(z,\z)\right)\le K_{\Psi}(z,\z)
    \le C_1\left(K_{\Psi}^{0}(z,\z)+K_{\Psi}^{1/2}(z,\z)\right),\quad z,\z\in\D.
    \end{equation}
\end{lemma}

\begin{proof}
Let $\b\in\{0,1/2\}$ and $z,\z\in\D$. If both $z$ and $\z$ are distinct from zero, choose $I_{0}=I_{0}(z,\z)\in\mathcal D^{\beta}$
 of minimal length such that $|I_{0}|\ge\max\{1-|z|,1-|\z|\}$ and $z/|z|,\z/|\z|\in I_{0}$, for otherwise, take $I_0=I^\b_{0,0}$. Then $z,\z\in S(I_0)$.
Let $N\in\N$ such that $2^N|I_0|=4$. Since $\Psi$ is essentially
decreasing by the hypothesis, we deduce
    \begin{equation}
    \begin{split}
    \sum_{I\in \mathcal D^{\beta}} \frac{1_{S(I)}(z)1_{S(I)}(\z)\Psi(|I|)}{|I|}
    &=\sum_{I\in \mathcal D^{\beta}, I_{0}\subset I}\frac{\Psi(|I|)}{|I|}
    = \sum_{k=0}^N \frac{\Psi(2^k|I_0|)}{2^k|I_0|}\\
    &\lesssim \frac{\Psi(|I_0|)}{|I_0|}\sum_{k=0}^N \frac{1}{2^k}
    \lesssim \frac{\Psi(|I_0|)}{|I_0|}.
    \end{split}
    \end{equation}
A direct calculation shows that $|1-\overline{\z}z|\le C|I_{0}|$ for some $C>1$. As $\Psi$ is essentially decreasing and admits the doubling property, we obtain
    $$
    \sum_{I\in \mathcal D^{\beta}} \frac{1_{S(I)}(z)1_{S(I)}(\z)\Psi(|I|)}{|I|}
    \lesssim\frac{\Psi\left(\frac{|1-\overline{\z}z|}{C}\right)}{|1-\overline{\z}z|}
    \lesssim \frac{\Psi\left(|1-\overline{\z}z|\right)}{|1-\overline{\z}z|}.
    $$
Since $\b$ was either 0 or 1/2, the left-hand inequality in \eqref{eq:kernelcomp} is proved.

To prove the right hand inequality, let $z,\z\in\D$. Let
$J=J(z,\z)\subset\T$ such that $z,\z\in S(J)$ and
$|J|\asymp|1-\overline{\z}z|$, see~\cite{AlPoRe} for details. There
exist $\beta\in \{0, 1/2\}$ and $K\in\mathcal D^{\beta}$ such that
$J\subset K$ and $|K|\le4 |J|$. By using the hypotheses on $\Psi$,
we get
    \begin{equation*}
    \begin{split}
    \frac{\Psi(|1-\overline{\z}z|)}{|1-\overline{\z}z|}
    &\lesssim \frac{\Psi(|J|)}{|J|}
    \lesssim \frac{\Psi(|K|)}{|K|}
    \lesssim\sum_{\substack{I\in\mathcal D^{\beta}\\ K\subset J}}\frac{1_{S(I)}(z)1_{S(I)}(\z)\Psi(|I|)}{|I|}
    \lesssim K_{\Psi}^{0}(z,\z)+K_{\Psi}^{1/2}(z,\z),
    \end{split}
    \end{equation*}
and the lemma is proved.
\end{proof}

For a positive Borel measure $\nu$ and a dyadic grid $\mathcal D$ on $\T$, the dyadic weighted Hardy-Littlewood (or H\"ormander type) maximal function is defined as
    \begin{equation}\label{e.max}
    M_{\nu,\mathcal D}(f)(z)= \sup_{I\in\mathcal D} \frac{1_{S(I)}(z)}{\nu(S(I))}\int_{S(I)}|f(\z)|\,d\nu(\z).
    \end{equation}
The maximal operator $M_{\nu,\mathcal D^\beta}$ appears naturally in the study of the dyadic operator $P^{\beta}_{\Psi,\mu}$. Its standard boundedness properties are given in the next lemma.

\begin{letterlemma}\label{Lemma:maximal-dyadic}
Let $\nu$ be a positive Borel measure and $\mathcal D$ a dyadic grid on $\T$. Then $M_{\nu,\mathcal D}:L^1_\nu\to L^{1^\infty}_\nu$ is bounded and consequently, $M_{\nu,\mathcal D}:L^p_\nu\to L^p_\nu$ is bounded for each $1<p\le\infty$. In particular, there exists a constant $C=C(p)>0$ such that
    \begin{equation}\label{e.mbounded}
    \|M_{\nu,\mathcal D}(f)\|_{L^{p}_\nu}\le C \| f\|_{L^{p}_\nu}.
    \end{equation}
\end{letterlemma}

\begin{proof}
By the Marcinkiewicz interpolation theorem it is enough to prove the weak (1,1) inequality.
Let $f\ge 0$, $\a>0$ and $O_\a=\{z\in\D: M_{\nu,\mathcal D}f(z)>\a\}$. Further, let $\Phi$ be the family of Carleson squares $S\in\mathcal D$ such that
    $$
    \int_{S}|f(\z)|\,d\nu(\z)>\a\nu(S),
    $$
and let $\Phi^{\text{max}}$ the subfamily of $\Phi$ consisting of the maximal Carleson squares.
Then $\Phi^{\text{max}}$ is a covering of $O_\a$ and each $z\in O_\a$ is contained in at most two different squares in $\Phi^{\text{max}}$. Therefore
    $$
    \nu\left(O_\a \right)\le\sum_{S\in \Phi^{\text{max}}}\nu(S)\le \frac{1}{\a}\sum_{S\in \Phi^{\text{max}}}\int_{S}|f(\z)|\,d\nu(\z)
    \le\frac{2}{\a}\int_{O_\a}|f(\z)|\,d\nu(\z)
    \le\frac{2}{\a} \|f\|_{L^1_\nu},
    $$
and the lemma is proved.
\end{proof}

Let $v,u\in L^1_\mu$ non-negative, and let $1<p<\infty$ and $p'$ its dual exponent. The dual weight of~$v$ is $\sigma=\sigma(p,v)=v^{1-p'}$. If $T$ is a linear operator, the following are equivalent:
    \begin{itemize}
    \item[\rm(A)]$T:L_\mu^{p}(v)\to L_\mu^{p}(u)$ is bounded;
    \item[\rm(B)]$T(\sigma\cdot): L_\mu^{p}(\sigma)\to L_\mu^{p}(u)$ is bounded;
    \item[\rm(C)]$u^{1/p}T(\sigma^{1/p'}\cdot):L_\mu^{p}\to L_\mu^{p}$ bounded.
    \end{itemize}
Moreover,
    \begin{equation}\label{equivnorms}
    \|T\|_{L_\mu^{p}(v)\to L_\mu^{p}(u)}=\|T(\sigma\cdot)\|_{L_\mu^{p}(\sigma)\to L_\mu^{p}(u)}=
    \|u^{1/p}T(\sigma^{1/p'}\cdot)\|_{L_\mu^{p}\to L_\mu^{p}}.
    \end{equation}

We now show how to obtain a linear bound for our dyadic operator in terms of the $B_{p,\mu}$-characteristic. This requires some hypotheses on the measure $\mu$ and the function $\Psi$. The following theorem is an extension of the main result of~\cite{PottRegueraJFA13}.

\begin{theorem}\label{th:sufp>1}
Let $1<p<\infty$, $\mu$ a positive Borel measure on $\D$ and $v\in B_{p,\mu}$. Let $\Psi:D(1,1)\to\mathbb{C}$ be an analytic function such that its restriction to the interval $(0,2)$ is positive and $|\Psi(1-z)|=|\Psi(1-\overline{z})|$ for all $z\in\D$.
Further, assume that $\mu(S(I))\lesssim\mu(T(I))$ and $\Psi(\vert I\vert)
\mu(S(I))\lesssim|I|$ for all dyadic intervals $I$. Then
    $$
    \left\Vert P^{\beta}_{\Psi,\mu}(f)\right\Vert_{L_\mu^p(v)}\lesssim B_{p,\mu}(v)^{\max\left\{1,\frac{1}{p-1}\right\}}
    \left\Vert f\right\Vert_{L_\mu^p(v)},\quad \beta\in\left\{0,\frac{1}{2}\right\}.
    $$
\end{theorem}

\begin{proof}
We focus first on the case $p=2$ since it is easiest.  We then
explain how to either obtain the result for all $p$ from this or how
to modify the proof given to provide a direct proof for all $p$.

We will proceed by duality to study the norm of
$P^{\beta}_{\Psi,\mu}(v^{-1}\cdot):L^2_\mu(v^{-1})\to L_\mu^2(v)$. Then the assertion for $p=2$ follows by \eqref{equivnorms}. Suppose that $f\in L^2_\mu(v^{-1})$ and $g\in L^2_\mu(v)$ are non-negative functions.
The role of $\beta$ now plays no role and so we drop its dependence.
Then
    \begin{align*}
    &\int_{\mathbb{D}} P_{\Psi,\mu}(v^{-1}f)(z)g(z)v(z)\,d\mu(z)\\
    &=\sum_{I\in\mathcal{D}}\left\langle v^{-1}f,1_{S(I)}\right\rangle_{L_\mu^2} \left\langle vg, 1_{S(I)}\right\rangle_{L_\mu^2}\frac{\Psi(\left\vert I\right\vert)}{\left\vert I\right\vert}\\
    &=\sum_{I\in \mathcal D}
    \frac{\Psi(\left\vert I\right\vert)\mu(S(I))^2}{\left\vert I\right\vert}\left( \frac{\int_{S(I)} f v^{-1}\,d\mu}{\int_{S(I)} v^{-1}\,d\mu}\right)
    \left(\frac{\int_{S(I)} g v\,d\mu}{\int_{S(I)} v\,d\mu}\right)
    \frac{\int_{S(I)}v^{-1}\,d\mu}{\mu(S(I))}
    \frac{\int_{S(I)}v\,d\mu}{\mu(S(I))}\\
    &\lesssim B_{2,\mu}(v)\sum_{I\in\mathcal D}\mu(T(I))\left(\frac{\int_{S(I)}fv^{-1}\,d\mu}{\int_{S(I)}v^{-1}\,d\mu}\right)
    \left(\frac{\int_{S(I)}gv\,d\mu}{\int_{S(I)}v\,d\mu}\right)\\
    &=B_{2,\mu}(v)\sum_{I\in\mathcal D}\int_{T(I)}
    \left(\frac{\int_{S(I)}fv^{-1}\,d\mu}{\int_{S(I)}v^{-1}\,d\mu}\right)
    \left(\frac{\int_{S(I)}gv\,d\mu}{\int_{S(I)}v\,d\mu}\right)\,d\mu(z)\\
    &\le B_{2,\mu}(v)\int_{\mathbb D}\left(M_{v^{-1}\mu,\mathcal D}(f)(z)v^{-\frac12}(z)\right)
    \left(M_{v\mu,\mathcal D}(g)(z)v^\frac12(z)\right)\,d\mu(z)\\
    &\lesssim B_{2,\mu}(v)\left\Vert f\right\Vert_{L^2_\mu(v^{-1})}
    \left\Vert g\right\Vert_{L^2_\mu(v)},
    \end{align*}
where the first inequality follows from the hypotheses on $\mu$, $\Psi$ and $v$; the second by the domination of the averages by the maximal
functions; and the last by the Cauchy-Schwarz inequality and the boundedness of the maximal functions due to~Lemma~\ref{Lemma:maximal-dyadic}.

It is possible to use the standard extrapolation proof to show that
this estimate can be lifted to $1<p<\infty$ with an appropriate
change in the characteristic for the weight $v$; see
\cite{PottRegueraJFA13} for these details. It is instead possible to
provide a direct proof by using a verbatim repetition of the
proof above. We sketch the modifications now and leave the details to
the reader.

Consider first the case $1<p\le2$. Let $\sigma=v^{1-p'}$. The goal is
to now prove that
    \begin{align*}
    \left\Vert P_{\Psi,\mu}(\sigma f)\right\Vert_{L^p_\mu(v)}
    \lesssim B_{p,\mu}(v)^{\frac{1}{p-1}}
    \left\Vert f\right\Vert_{L^p_\mu(\sigma)}.
    \end{align*}
It is more convenient to prove the equivalent inequality
    \begin{align*}
    \left\Vert P_{\Psi,\mu}(\sigma f)^{p-1}\right\Vert_{L^{p'}_\mu(v)}
    \lesssim B_{p,\mu}(v) \left\Vert
    f\right\Vert_{L^{p}_\mu(\sigma)}^{p-1}.
    \end{align*}
This last inequality can be studied via duality as above.  Since
$1<p\leq 2$, and the function $h(x)=x^{r}$ is sub-additive for
$0<r<1$ we obtain
    \begin{equation*}
    \begin{split}
    \left\langle P_{\Psi,\mu}(\sigma f)^{p-1}, vg\right\rangle_{L_\mu^2}
    &\le\sum_{I\in \mathcal{D}}\left\langle \sigma f, 1_{S(I)}\right\rangle_{L_\mu^2}^{p-1} \left(\frac{\Psi(\left\vert I\right\vert)}{\left\vert I\right\vert}\right)^{p-1} \left\langle vg, 1_{S(I)}\right\rangle_{L_\mu^2}\\
    &\lesssim B_{p,\mu}(v)\left\Vert M_{\sigma\mu}(f)\right\Vert_{L_\mu^{p}(\sigma)}^{p-1}\left\Vert
    M_{v\mu}(g)\right\Vert_{L_\mu^{p}(v)}.
    \end{split}
    \end{equation*}
The inequality above is obtained exactly as above in the case $p=2$
by using the definition of $B_{p,\mu}(v)$, the relationship
between $\mu$ and $\Psi$. Estimates of the maximal function then
provide the desired estimates to control the duality. The case
$2<p<\infty$ can be deduced via the self-adjointness of $P_{\Psi,
\mu}$ with respect to $\langle\cdot,\cdot\rangle_{L^2_\mu}$, the result for $1<p<2$ and the relationship between
$B_{p,\mu}(v)$ and $B_{p',\mu}(v)$.
\end{proof}

Because of the equivalence we have between the dyadic operators $P^\b_{\Psi,\mu}$ and
and $P^+_{\Psi,\mu}$ given in Lemma~\ref{p.pointwise}, we obtain the following
result.

\begin{corollary}\label{co:sufp>1}
Let $1<p<\infty$, $\mu$ a positive Borel measure on $\D$ and $v\in B_{p,\mu}$. Let $\Psi:D(1,1)\to\mathbb{C}$ be an analytic function such that its restriction to the interval $(0,2)$ is positive and essentially decreasing, $\Psi(t)\lesssim\Psi(2t)$ for all $t\in(0,1)$, and $|\Psi(1-z)|=|\Psi(1-\overline{z})|$ for all $z\in\D$.
Further, assume that $\mu(S(I))\lesssim\mu(T(I))$ and $\Psi(\vert I\vert)
\mu(S(I))\lesssim|I|$ for all intervals $I$. Then
    $$
    \left\Vert P^+_{\Psi,\mu}\right\Vert_{L_\mu^p(v)\to L_\mu^p(v)}
    \lesssim B_{p,\mu}(v)^{\max\left\{1,\frac{1}{p-1}\right\}}.
    $$
\end{corollary}

The upper bound for the operator norm given in Corollary~\ref{co:sufp>1} is essentially independent of~$\Psi$, and therefore it is not necessarily sharp for all admissible $\Psi$. But when we apply it in the proof of Theorem~\ref{Theorem:BB-generalized} to deduce that $v\in B_{p,\om}$ is a sufficient condition for $P^+_\om:L^p_{\omega}(v)\to L^p_{\omega}(v)$ to be bounded, the hypotheses on $\Psi$ and $\om$ in question are satisfied precisely, meaning that $\lesssim$ are in fact $\asymp$, and hence the resulting sufficient condition will also be necessary. This will be discussed in more detail at the end of the section when the proof of Theorem~\ref{Theorem:BB-generalized} is finally pulled together.

We next proceed with auxiliary results needed to show that $v\in B_{p,\om}$ is a necessary condition for $P_\om:L^p_{\omega}(v)\to L^{p,\infty}_{\omega}(v)$ to be bounded.

\begin{lemma}\label{Lemma:pointwise}
Let $\om$ be a positive Borel measure such that $B^\om_z$ admits the
representation
    $$
    B^\om_z(\z)=\frac{1}{(1-\overline{z}\z)^\gamma}\int_0^1\frac{d\nu(r)}{1-r\overline{z}\z},\quad z,\z\in\D,
    $$
for some $\gamma\ge1$ and a positive measure $\nu$ supported on
$[0,1]$, and let $c>1$. Then
    $$
    |B^\om_{z_0}(\z)-B^\om_{z}(\z)|\le C\frac{|z-z_0|}{|1-\overline{\z}z|}|B^\om_z(\z)|
    $$
for all $z,z_0,\z\in\D$ with $|1-\overline{\z}z|\ge c|z-z_0|$, where
    $$
    C=C(c,\gamma)=\sqrt{2}(2+\gamma)\frac{c^{\gamma+1}(3c+1)}{(c-1)^{\gamma+2}}\to3\sqrt{2}(2+\gamma),\quad c\to\infty.
    $$
\end{lemma}

\begin{proof}
A direct calculation shows that
    $$
    |B^\om_{z_0}(\z)-B^\om_{z}(\z)|=|B^\om_{\z}(z_0)-B^\om_{\z}(z)|=\left|\int_{z_0}^z(B^\om_\z)'(x)\,dx\right|\le\int_{z_0}^z|(B^\om_\z)'(x)||dx|,
    $$
where
    $$
    B_\z'(x)=\frac{\gamma\overline{\z}}{(1-\overline{\z}x)^{\gamma+1}}\int_0^1\frac{d\nu(r)}{1-r\overline{\z}x}
    +\frac{\overline{\z}}{(1-\overline{\z}x)^{\gamma}}\int_0^1\frac{rd\nu(r)}{(1-r\overline{\z}x)^2},\quad \z,x\in\D,
    $$
and hence
    $$
    |B_\z'(x)|\le\frac{\gamma|\z|}{|1-\overline{\z}x|^{\gamma+1}}\int_0^1\frac{d\nu(r)}{|1-r\overline{\z}x|}
    +\frac{|\z|}{|1-\overline{\z}x|^{\gamma}}\int_0^1\frac{d\nu(r)}{|1-r\overline{\z}x|^2},\quad \z,x\in\D.
    $$
Since $|1-w|\le2|1-rw|$ for all $w\in\D$ and $0\le r\le1$, we deduce
    $$
    |B_\z'(x)|\le\frac{(2+\gamma)|\z|}{|1-\overline{\z}x|^{\gamma+1}}\int_0^1\frac{d\nu(r)}{|1-r\overline{\z}x|},\quad \z,x\in\D.
    $$
It follows that
    $$
    |B^\om_{z_0}(\z)-B^\om_{z}(\z)|\le(2+\gamma)|\z||z-z_0|
    \sup_{x\in[z,z_0]}\left(\frac{1}{|1-\overline{\z}x|^{\gamma+1}}\int_0^1\frac{d\nu(r)}{|1-r\overline{\z}x|}\right).
    $$
If $x\in[z,z_0]$, then $|1-\overline{\z}z|\ge c|z-z_0|\ge c|z-x|$,
and hence
    $$
    |1-\overline{\z}x|=|1-\overline{\z}z+\overline{\z}z-\overline{\z}x|\ge|1-\overline{\z}z|-|\z||z-x|
    \ge|1-\overline{\z}z|-c^{-1}|1-\overline{\z}z|=\left(1-\frac1c\right)|1-\overline{\z}z|.
    $$
Thus
    $$
    |B^\om_{z_0}(\z)-B^\om_{z}(\z)|\le\frac{(2+\gamma)|\z||z-z_0|}{\left(1-\frac1c\right)^{\gamma+1}|1-\overline{\z}z|^{\gamma+1}}
    \sup_{x\in[z,z_0]}\left(\int_0^1\frac{d\nu(r)}{|1-r\overline{\z}x|}\right).
    $$
Let $\d\in(0,1)$. Then
    \begin{equation*}
    \int_0^1\frac{d\nu(r)}{|1-r\overline{\z}x|}\le\frac1{1-\d}\nu([0,1])+\int_\d^1\frac{d\nu(r)}{|1-r\overline{\z}x|}
    \le\frac{2}{1-\d}\int_0^1\frac{d\nu(r)}{|1-r\overline{\z}z|}+\int_\d^1\frac{d\nu(r)}{|1-r\overline{\z}x|}.
    \end{equation*}
A direct calculation or a geometric reasoning shows that
$|1-w|\le\frac{2}{1+\d}|1-rw|$ for all $w\in\D$ and $\d\le r\le1$.
Hence
    \begin{equation*}
    |1-r\overline{\z}x|\ge|1-r\overline{\z}z|-c^{-1}|1-\overline{\z}z|\ge|1-r\overline{\z}z|-\frac{2}{c(1+\d)}|1-r\overline{\z}z|,\quad \d\le r\le1.
    \end{equation*}
By choosing $\d=1/c$, we deduce
    $$
    |1-r\overline{\z}x|\ge\frac{c-1}{c+1}|1-r\overline{\z}z|,
    $$
and it follows that
    \begin{equation*}
    \int_0^1\frac{d\nu(r)}{|1-r\overline{\z}x|}\le\frac{3c+1}{c-1}\int_0^1\frac{d\nu(r)}{|1-r\overline{\z}z|}.
    \end{equation*}
Since
    \begin{equation}\label{4}
    \begin{split}
    \left|\int_0^1 \frac{d\nu(r)}{1-rz}\right|  & = \left( \left(\int_0^1 \frac{(1-r|z|\cos(\theta))d\nu(r)}{|1-rz|^2}\right)^2+     \left(\int_0^1 \frac{r|z|\sin(\theta)d\nu(r)}{|1-rz|^2}\right)^2\right)^{1/2}\\
    &\ge \frac{1}{\sqrt{2}}\int_0^1 \frac{\left[(1-r|z|\cos(\theta))+ r|z||\sin(\theta)|\right]d\nu(r)}{|1-rz|^2}\\
    &\ge \frac{1}{\sqrt{2}}\int_0^1 \frac{d\nu(r)}{|1-rz|},\quad z\in\D,
    \end{split}
    \end{equation}
we deduce
    \begin{equation*}
    \begin{split}
    |B^\om_{z_0}(\z)-B^\om_{z}(\z)|
    &\le\frac{(2+\gamma)|\z||z-z_0|}{\left(1-\frac1c\right)^{\gamma+1}|1-\overline{\z}z|^{\gamma+1}}
    \frac{3c+1}{c-1}\int_0^1\frac{d\nu(r)}{|1-r\overline{\z}z|}\\
    &\le\frac{\sqrt{2}(2+\gamma)|\z||z-z_0|}{\left(1-\frac1c\right)^{\gamma+1}|1-\overline{\z}z|^{\gamma+1}}
    \frac{3c+1}{c-1}\left|\int_0^1\frac{d\nu(r)}{1-r\overline{\z}z}\right|
    \le C\frac{|z-z_0|}{|1-\overline{\z}z|}|B^\om_z(\z)|,
    \end{split}
    \end{equation*}
where $C=C(c,\gamma)=\sqrt{2}(2+\gamma)\frac{c^{\gamma+1}(3c+1)}{(c-1)^{\gamma+2}}$.
\end{proof}

\begin{lemma}\label{le:shi1}
Let $\om\in\DD$ such that $B^\om_z$ admits the representation
    \begin{equation}\label{Eq:kernel-representation-repeated}
    B^\om_z(\z)=\frac{1}{(1-\overline{z}\z)^\gamma}\int_0^1\frac{d\nu(r)}{1-r\overline{z}\z},\quad z,\z\in\D,
    \end{equation}
for some $\gamma\ge1$ and a positive measure $\nu$ supported on
$[0,1]$. Then
    \begin{equation}\label{eq:shi1}
    \int_0^1\frac{d\nu(r)}{1-rx}\asymp \frac{(1-x)^{\gamma-1}}{\widehat{\om}(x)},\quad x\in[0,1).
    \end{equation}
\end{lemma}

\begin{proof}
By \cite[Lemmas 3.1 and 3.2]{PelSum14}, see also \cite[Lemma~6.2]{PelRat},
    $$
    \frac{1}{(1-|z|^2)\widehat{\om}(|z|^2)}
    \asymp\|B^\om_z\|^2_{A^2_\om}
    =B^\om_z(z)
    =\frac{1}{(1-|z|^2)^\gamma}\int_0^1\frac{d\nu(r)}{1-r|z|^2},\quad z\in\D,
    $$
which is equivalent to \eqref{eq:shi1}.
\end{proof}

For a Carleson square $S=S(I)$, let $\ell(S)=|I|$ denote its side length.

\begin{lemma}\label{le:poinwisestimate}
Let $\om\in\DD$ such that $B^\om_z$ admits the representation \eqref{Eq:kernel-representation-repeated}
for some $\gamma\ge1$ and a positive measure $\nu$ supported on
$[0,1]$. Then there are constants $D_1=D_1(\gamma)>0$ and $D_2=D_2(\gamma)>0$ such that
for all (sufficiently small) Carleson squares $S_1$ and $S_2$, with $\ell(S_1)=\ell(S_2)$ and $D_1
\ell(S_1)\le\dist(S_1,S_2)\le D_2\ell(S_1)$, we have
    \begin{equation}\label{eq:poinwisestimate}
    |P_\om(f)(z)|\ge C\frac{\int_{S_1}
    f(\z)\om(\z)\,dA(\z)}{\om(S_1)},\quad z\in S_2,
    \end{equation}
for some constant $C=C(D_1,D_2,\om)>0$ and for all nonnegative functions $f$ supported on $S_1$.
\end{lemma}

\begin{proof}
Let $S_1$ and $S_2$ be (small) Carleson squares such that $\ell(S_1)=\ell(S_2)$ and $D_1
\ell(S_1)\le\dist(S_1,S_2)\le D_2\ell(S_1)$, where $D_1,D_2>0$ are absolute constants to be fixed later. Let $\z_0$ be the center of $S_1$. Then
    \begin{equation}\label{eq:p1}
    |P_\om(f)(z)|
    \ge|B^\om_{\z_0}(z)|\int_{S_1}f(\z)\om(\z)\,dA(\z)-\int_{S_1}f(\z)|B^\om_{\z}(z)-B^\om_{\z_0}(z)| \om(\z)\,dA(\z)
    \end{equation}
for all $z\in\D$. If $z\in S_2$ and $\z\in S_1$, then
    \begin{equation}\label{3}
    \begin{split}
    |1-\overline{z}\z_0|
    &\ge|z-\z_0|
    \ge\frac{\ell(S_1)}{3}+\textrm{dist}(S_1,S_2)
    \ge \left(\frac{1}{3}+D_1 \right)\ell(S_1)
    \ge c_1|\z-\z_0|,
    \end{split}
    \end{equation}
where $c_1=\frac{\left(\frac{1}{3}+D_1 \right)}{\sqrt{2}}$.
Choose $D_1=D_1(\gamma)>1$ sufficiently large such that $c_1>1$ and
$\sqrt{2}(2+\gamma)\frac{c_1^{\gamma}(3c_1+1)}{(c_1-1)^{\gamma+2}}\le
\frac{1}{2}$.
Then, by using Lemma~\ref{Lemma:pointwise} and \eqref{3}, we deduce
    \begin{equation}
    \begin{split}\label{eq:p2}
    |B^\om_{\z}(z)-B^\om_{\z_0}(z)|
    &\le\sqrt{2}(2+\gamma)\frac{c_1^{\gamma+1}(3c_1+1)}{(c_1-1)^{\gamma+2}}\frac{|\z_0-\z|}{|1-\z_0\overline{z}|}|B^\om_{\z_0}(z)|\\
    &\le\sqrt{2}(2+\gamma)\frac{c_1^{\gamma}(3c_1+1)}{(c_1-1)^{\gamma+2}}|B^\om_{\z_0}(z)|
    \le\frac{1}{2}|B^\om_{\z_0}(z)|.
    \end{split}
    \end{equation}
By combining \eqref{eq:p1} and \eqref{eq:p2} we get
    \begin{equation}\label{eq:p3}
    |P_\om(f)(z)|\ge
    \frac{1}{2}|B^\om_{\z_0}(z)|\int_{S_1}f(\z)\om(\z)\,dA(\z),\quad z\in S_2.
    \end{equation}
Now, we observe that
    \begin{equation*}
    \begin{split}
    |1-\overline{\z}_0z|
    &\le(1-|\z_0|^2)+|\z_0-z|)
    \le3\ell(S_1)+\dist(S_1,S_2)
    \le(3+D_2)\ell(S_1),\quad z\in S_2.
    \end{split}
    \end{equation*}
This together with \eqref{4}, the inequality $(a+xb)\le x(a+b)$ for $a,b>0$ and $x\ge1$, and Lemma~\ref{le:shi1} yield
    \begin{equation*}
    \begin{split}
    |B^\om_{\z_0}(z)|
    &\ge\frac{1}{\sqrt{2}|1-\overline{\z}_0z|^\gamma}\int_0^1\frac{d\nu(r)}{1-r(1-|1-\overline{\z}_0z|)}\\
    &\ge\frac{1}{\sqrt{2} (3+D_2)^\gamma\ell(S_1)^\gamma}\int_0^1\frac{d\nu(r)}{1-r+r(3+D_2)\ell(S_1)}\\
    &\ge\frac{1}{\sqrt{2} (3+D_2)^{\gamma+1}\ell(S_1)^\gamma}\int_0^1\frac{d\nu(r)}{1-r+r\ell(S_1)}
    \ge\frac{C}{\sqrt{2} (3+D_2)^{\gamma+1}\om(S_1)}
    \end{split}
    \end{equation*}
for some constant $C=C(\om)>0$. The assertion follows by combining this with \eqref{eq:p3}.
\end{proof}

\begin{proposition}\label{pr:weakpp>1}
Let $1<p<\infty$, $\om\in\DD$ such that $B^\om_z$ admits the representation \eqref{Eq:kernel-representation-repeated} for some $\gamma\ge1$ and a positive measure $\nu$ supported on
$[0,1]$, and $v\in L^1_{\om,\rm{loc}}$ non-negative. If $P_\om: L^p_{\om}(v)\to
L^{p,\infty}_{\om}(v)$ is bounded, then $v\in B_{p,\omega}$.
\end{proposition}

\begin{proof}
It suffices to show that the quantity
    $$
    \frac{(v\om)(S)}{\om(S)}
    \left(\frac{\left(v^{-\frac{p'}{p}}\om\right)(S)}{\om(S)}\right)^{\frac{p}{p'}}
    $$
is uniformly bounded for all small Carleson squares $S$. By the hypothesis, there
exists $C_1>0$ such that
    \begin{equation}\label{eq:wpp1}
    \lambda^p(v\om)\left(\left\{z\in\D: |P_\om(f)(z)|\ge\lambda
    \right\}\right)\le C_1\|f\|^p_{L^p_{\om}(v)}, \quad \lambda>0.
    \end{equation}
Let $S_1$ be a sufficiently small Carleson square, and choose
    $$
    \lambda=C\frac{\int_{S_1}\left(\min\{n,v^{-\frac{p'}{p}}(\z)\}\right)\dm(\z)}{\om(S_1)},\quad n\in\N,
    $$
where $C$ is the constant appearing in~\eqref{eq:poinwisestimate}. Further, choose $f=1_{S_1} \min\{n,v^{-\frac{p'}{p}}\}$. Then we get
    \begin{equation*}
    \begin{split}
    &\frac{\left(\int_{S_1}\min\{n,v^{-\frac{p'}{p}}(\z)\}\dm(\z)\right)^{p-1}}{\om(S_1)^p}\\
    &\cdot(v\om)\left(\left\{z\in\D:|P_\om(f)(z)|\ge C\frac{\int_{S_1}\min\{n,v^{-\frac{p'}{p}}(\z)\}\dm(\z)}{\om(S_1)}\right\}\right)\le\frac{C_1}{C^p}.
    \end{split}
    \end{equation*}
By Lemma~\ref{le:poinwisestimate}, for all suitable $S_2$ with $\ell(S_2)=\ell(S_1)$ we have
    $$
    S_2\subset\left\{z\in\D:|P_\om(f)(z)|\ge C\frac{\int_{S_1}\min\{n,v^{-\frac{p'}{p}}(\z)\}\dm(\z)}{\om(S_1)}\right\},
    $$
and it follows that
    $$
    \frac{\left(\int_{S_1}\min\{n,v^{-\frac{p'}{p}}(\z)\}\dm(\z)\right)^{p-1}\int_{S_2}v(\z)\dm(\z)}{\om(S_1)^p}\le \frac{C_1}{C^p}.
    $$
By changing the roles of $S_1$ and $S_2$ we deduce
    $$
    \frac{\left(\int_{S_2}\min\{n,v^{-\frac{p'}{p}}(\z)\}\dm(\z)\right)^{p-1}\int_{S_1}v(\z)\dm(\z)}{\om(S_2)^p}\le \frac{C_1}{C^p},
    $$
and it follows that $v\in L^1_{\om}$. By letting $n\to\infty$ and using Fatou's lemma we deduce
    $$
    \frac{\left(\int_{S_1}v^{-\frac{p'}{p}}\dm\right)^{p-1}\int_{S_2}v\dm}{\om(S_1)^p}
    \frac{\left(\int_{S_2}v^{-\frac{p'}{p}}\dm\right)^{p-1}\int_{S_1}v\dm}{\om(S_2)^p}\le\frac{C_1^2}{C^{2p}}.
    $$
Since
    $$
    \frac{(v\om)(S)}{\om(S)^p}
    \left(\left(v^{-\frac{p'}{p}}\om\right)(S)\right)^{\frac{p}{p'}}
    =\frac{(v\om)(S)}{\om(S)}
    \left(\frac{\left(v^{-\frac{p'}{p}}\om\right)(S)}{\om(S)}\right)^{\frac{p}{p'}}\ge1
    $$
for any Carleson square $S$ by H\"older's inequality, it follows that $v\in B_{p,\omega}$.
\end{proof}

With these preparations we are ready to prove the first of our main results.

\medskip

\begin{Prf}{\em{Theorem~\ref{Theorem:BB-generalized}}.}
Clearly, (i)$\Rightarrow$(ii)$\Rightarrow$(iii), and (iii)$\Rightarrow$(iv) follows by Proposition~\ref{pr:weakpp>1}. To see the remaining implication, note that
    $$
    B^\om_z(\z)=\frac1{(1-\overline{z}\z)^{\gamma}}\int_0^1\frac{d\nu(r)}{1-r\overline{z}\z}=\frac{\Psi(1-\overline{z}\z)}{1-\overline{z}\z}
    $$
for the analytic function
    $$
    \Psi(z)=z^{1-\gamma}\int_0^1\frac{d\nu(r)}{1-r(1-z)},\quad z\in D(1,1).
    $$
The restriction of $\Psi$ to $(0,2)$ is decreasing because $\gamma\ge1$, and obviously $|\Psi(1-z)|=|\Psi(1-\overline{z})|$ for all $z\in\D$. Moreover, $\mu=\omega\otimes m$ satisfies $\mu(S(I))\lesssim\mu(T(I))$ because $\om\in\R$, and Lemma~\ref{le:shi1} yields
    $$
    \Psi(|I|)=\frac1{|I|^{\gamma-1}}\int_0^1\frac{d\nu(r)}{1-r(1-|I|)}\asymp\frac1{\widehat{\om}(1-|I|)},
    $$
so $\Psi(|I|)\mu(S(I))\asymp|I|$ for all intervals $I$. Now that $\Psi(t)\lesssim\Psi(2t)$ for all $t\in(0,1)$, the hypothesis of Corollary~\ref{co:sufp>1} are satisfied, and hence (iv)$\Rightarrow$(i) as well as the estimate for the operator norm of $P^+_\om$ follow.
\end{Prf}

\section{Weak type $(1,1)$ inequality}\label{Sec4}

\begin{lemma}\label{Lemma:nu}
Let $\nu$ be a positive Borel measure supported on $[0,1]$. Then
    \begin{equation*}
    \left|\int_0^1 \frac{d\nu(r)}{1-rz}\right|\asymp \int_0^1 \frac{d\nu(r)}{\left|1-rz\right|}\asymp \int_0^1 \frac{d\nu(r)}{1-r(1-|1-z|)},\quad z\in\D.
    \end{equation*}
\end{lemma}

\begin{proof}
We first show that
    \begin{equation}\label{equiv1}
    1-r(1-|1-z|)\asymp|1-rz|, \quad z\in\D.
    \end{equation}
On one hand, $|1-rz|=|1-r+r(1-z)|\le (1-r)+r|1-z|$ for all $0<r<1$. On the other
hand, if $z=|z|e^{i\theta}$ and $r\ge 1/2$, then
    \begin{equation*}
    \begin{split}
    |1-rz|^2 &=\left( (1-r)+r(1-|z|) \right)^2 +4r|z|\sin^2\left(\frac{\theta}{2}\right)\\
    &\ge\frac{1}{4}\left((1-r)^2+(1-|z|)^2\right)+4r|z|\sin^2\left(\frac{\theta}{2}\right)\\
    &\ge\frac{1}{4}\left((1-r)^2+(1-|z|)^2+4|z|\sin^2\left(\frac{\theta}{2}\right)\right)\\
    &=\frac{1}{4}\left((1-r)^2+ |1-z|^2\right),
    \end{split}
    \end{equation*}
and hence
    $$
    |1-rz|\ge\frac{1}{2\sqrt{2}}\left((1-r)+|1-z|\right).
    $$
Moreover, for $0\le r\le 1/2$ we have
    $$
    |1-rz|\ge\frac{1}{2}\ge\frac{\left((1-r)+ |1-z|\right)}{6},
    $$
and hence \eqref{equiv1} follows. Therefore
    \begin{equation*}
    \left|\int_0^1 \frac{d\nu(r)}{1-rz}\right|\le \int_0^1 \frac{d\nu(r)}{|1-rz|}
    \asymp \int_0^1 \frac{d\nu(r)}{1-r(1-|1-z|)},\quad z\in\D.
    \end{equation*}
By combining this with \eqref{4} we deduce the assertion.
\end{proof}

\begin{lemma}\label{Lemma:local-integrated}
Let $\om\in\DD$ such that $B^\om_z$ admits the representation
    $$
    B^\om_z(\z)=\frac{1}{(1-\overline{z}\z)^\gamma}\int_0^1\frac{d\nu(r)}{1-r\overline{z}\z},\quad z,\z\in\D,
    $$
for some $\gamma\ge1$ and a positive measure $\nu$ supported on
$[0,1]$. Then for $v\in L^1_\om$ non-negative, $z_0\in \D\setminus
D(0,1/2)$ and \,$z\in\D$ satisfying $|z-z_0|\le c(1-|z_0|)$ for a
constant $c>0$, there exists $C=C(c,\gamma,\om)>0$ such that
    $$
    \int_{\D\setminus D(z_0,2|z-z_0|)}|B^\om_{z_0}(\z)-B^\om_{z}(\z)|v(\z)\dm(\z)\le
    C \inf_{a\in D(z_0,\sqrt{2}(1-|z_0|))\cap\D}M_\om(v)(a).
    $$
\end{lemma}

\begin{proof}
If $\z\in\D\setminus D(z_0,2|z-z_0|)$, then
$2|z-z_0|\le|z_0-\z|<|1-\overline{\z}z_0|$, and hence
    $$
    \int_{\D\setminus D(z_0,2|z-z_0|)}|B^\om_{z_0}(\z)-B^\om_{z}(\z)|v(\z)\dm(\z)
    \lesssim|z-z_0|\int_\D\frac{|B^\om_{z_0}(\z)|}{|1-\overline{\z}z_0|}v(\z)\dm(\z)
    $$
by Lemma~\ref{Lemma:pointwise}. Let $k_0\in\N$ such that $2^{k_0}\sqrt{2}(1-|z_0|)\le1<2^{k_0+1}\sqrt{2}(1-|z_0|)$. Let $E_{-1}=\emptyset$, $E_k=\{z\in\D:|1-\overline{z}_0z|\le
2^k\sqrt{2}(1-|z_0|)\}$ for $k=0,\ldots,k_0$, and $E_{k_0+1}=\D\setminus E_{k_0}$. Then, by Lemma~\ref{Lemma:nu},
    \begin{equation*}
    \begin{split}
    &(1-|z_0|)\int_\D\frac{|B^\om_{z_0}(\z)|}{|1-\overline{\z}z_0|}v(\z)\dm(\z)\\
    &\lesssim (1-|z_0|)\sum_{k=0}^{k_0+1} \int_{E_k\setminus E_{k-1}}\frac{1}{|1-\overline{\z}z_0|^{\gamma +1}}
    \left(\int_0^1\frac{d\nu(r)}{1-r(1-|1-\overline{\z}z_0|)}\right)
    v(\z)\dm(\z).
    \end{split}
    \end{equation*}
Further, Lemma~\ref{le:shi1} and the hypothesis $\om\in\DD$ give
    \begin{equation*}
    \begin{split}
    &(1-|z_0|)\sum_{k=0}^{k_0} \int_{E_k\setminus E_{k-1}}\frac{1}{|1-\overline{\z}z_0|^{\gamma +1}}
    \left(\int_0^1\frac{d\nu(r)}{1-r(1-|1-\overline{\z}z_0|)}\right)
    v(\z)\dm(\z)\\
    &\asymp(1-|z_0|)\sum_{k=0}^{k_0} \int_{E_k\setminus E_{k-1}}\frac{1}{|1-\overline{\z}z_0|^{2}\widehat{\om}(1-|1-\overline{\z}z_0|)}
    v(\z)\dm(\z)\\
    &\lesssim\sum_{k=0}^{k_0}\frac{1}{2^{2k}(1-|z_0|)\widehat{\om}(1-2^{k-1}\sqrt{2}(1-|z_0|))}(\om v)(E_k\setminus E_{k-1})\\
    &\le\sum_{k=0}^{k_0}\frac{1}{2^{2k}(1-|z_0|)\widehat{\om}(1-2^{k-1}\sqrt{2}(1-|z_0|))}(\om v)\left(D(z_0,2^k\sqrt{2}(1-|z_0|))\cap\D\right)\\
    &\lesssim\sum_{k=0}^{k_0}\frac{1}{2^{k}\om \left(D(z_0,2^k\sqrt{2}(1-|z_0|))\cap\D\right)}(\om v)\left(D(z_0,2^k\sqrt{2}(1-|z_0|))\cap\D\right)\\
    &\le\inf_{a\in D(z_0,\sqrt{2}(1-|z_0|))\cap\D}M_\om(v)(a) \sum_{k=0}^\infty \frac{1}{2^{k}}
    \asymp\inf_{a\in D(z_0,\sqrt{2}(1-|z_0|))\cap\D}M_\om(v)(a).
    \end{split}
   \end{equation*}
Furthermore, clearly
    \begin{equation*}
    \begin{split}
    &(1-|z_0|)\int_{E_{k_0+1}}\frac{1}{|1-\overline{\z}z_0|^{\gamma +1}}
    \left(\int_0^1\frac{d\nu(r)}{1-r(1-|1-\overline{\z}z_0|)}\right)v(\z)\dm(\z)\\
    &\lesssim\inf_{a\in
    D(z_0,\sqrt{2}(1-|z_0|))\cap\D}M_\om(v)(a),
    \end{split}
   \end{equation*}
which together with the previous estimate finishes the proof.
\end{proof}

Write $\D=\overline{D(0,\frac12)}\cup R_1\cup R_2$, where $R_1$ and $R_2$ are dyadic Carleson squares.

\begin{lemma}\label{Lemma:C-Z}
Let $\om\in\R$ such that $\om([a,b])\asymp\om\left(\left[a,\frac{a+b}2\right]\right)\asymp\om\left(\left[\frac{a+b}2,b\right]\right)$ for all $0\le a,b\le1$, $f\in L^1_\om$ and $\lambda>\|f\|_{L^1_\om}$. Let $R\in\{R_1,R_2\}$. Then there exist~$F$ and~$\Omega$ such that $R=F\cup\Omega$, $F\cap\Omega=\emptyset$ and
\begin{itemize}
\item[\rm(i)] $|f(z)|\le\lambda$ almost everywhere on $F$;
\item[\rm(ii)] $\Omega=\cup_k Q_k$, where $Q_k\subset R$ are dyadic polar rectangles;
\item[\rm(iii)] $\displaystyle\om(\Omega)\le\frac{\|f1_R\|_{L^1_\om}}{\lambda}$;
\item[\rm(iv)] There is a constant $C=C(\om)>0$ such that
$\displaystyle\lambda\le\frac{1}{\omega(Q_k)}\int_{Q_k}|f(z)|\dm(z)\le
C\lambda$.
\end{itemize}
The Calder\'on-Zygmund decomposition of $f1_R:R\to\C$ is $f1_R=g+b$, where
    \begin{equation*}
    g(z)=\left\{\begin{array}{ll}
    f(z),&\quad z\in F\\
    \frac{1}{\om(Q_k)}\int_{Q_k}f(\z)\dm(\z),&\quad z\in Q_k
    \end{array}\right..
    \end{equation*}
\end{lemma}

\begin{proof}
Write $R=Q_{1,0}$ and pick $Q_{1,0}$ if
    $$
    \frac{1}{\om(Q_{1,0})}\int_{Q_{1,0}}|f(\z)|\dm(\z)\ge\lambda.
    $$
If not, divide $Q_{1,0}$ into $Q_{k,0}$, $j=1,\ldots,4$, and pick those for which
    $$
    \frac{1}{\om(Q_{k,1})}\int_{Q_{k,1}}|f(\z)|\dm(\z)\ge\lambda.
    $$
Divide the non-selected ones and proceed. By re-naming the selected sets as $Q_k$ and defining $\Omega=\cup_k Q_k$ we have (ii).

(i) Let $F=R\setminus\Omega$. For almost every $z\in F$ and each $k\in\N\cup\{0\}$ there exists a unique dyadic polar rectangle $Q_j$ of generation $j$ such that $z\in Q_j$ and
    $$
    \frac{1}{\om(Q_{j})}\int_{Q_{j}}|f(\z)|\dm(\z)\le\lambda.
    $$
Then $\cap_j\overline{Q_j}=\{z\}$, and hence
    $$
    |f(z)|=\lim_{j\to\infty}\frac{1}{\om(Q_{j})}\int_{Q_{j}}|f(\z)|\dm(\z)
    $$
for almost every $z\in F$ by Lebesgue's differentiation theorem. It follows that $|f|\le\lambda$ almost everywhere on $F$.

(iii) Since $\om(Q_k)\le\frac1\lambda\int_{Q_{k}}|f(\z)|\dm(\z)$ for each $k$, we have
    $$
    \omega(\Omega)\le\sum_k\omega(Q_k)\le\frac1{\lambda}\sum_k\int_{Q_k}|f(\z)|\dm(\z)
    =\frac1{\lambda}\int_{\Omega}|f(\z)|\dm(\z)\le\frac{\|f\|_{L^1_\om}}{\lambda}.
    $$

(iv) Since $\omega(R)=c\omega(\D)$ for some constant $c>0$, we have
    $$
    \frac{1}{\om(R)}\int_{R}|f(\z)|\dm(\z)=\frac{\|f1_R\|}{c\om(\D)}<\frac{\lambda}{c\om(\D)}.
    $$
For each $Q_k\ne R$, there exists a non-selected dyadic polar
rectangle $Q'$ from the preceding generation such that $Q_k\subset
Q'$. Since $\om\in\R$ such that $\om([a,b])\asymp\om\left(\left[a,\frac{a+b}2\right]\right)\asymp\om\left(\left[\frac{a+b}2,b\right]\right)$ for all $0\le a,b\le1$ by the hypothesis, we deduce
    $$
    \lambda>\frac1{\om(Q')}\int_{Q'}|f(\z)|\dm(\z)
    \ge\frac1{C\om(Q_k)}\int_{Q_k}|f(\z)|\dm(\z)
    $$
for some constant $C=C(\om)>0$, and thus (iv) holds.
\end{proof}

\medskip

\begin{Prf}{\em{Theorem~\ref{th:w11}}.}
Assume first that  $v$ satisfies the $B_{1,\om}$-condition. Write
$\D=\overline{D(0,\frac12)}\cup R_1\cup R_2$ as before. Then
$f=f1_{\overline{D(0,\frac12)}}+f1_{R_1}+f1_{R_2}$. Since $B^\om_z$
is uniformly bounded on $D(0,\frac12)$ and
    $$
    \essinf_{z\in D(0,\frac12)}v(z)\ge\frac{1}{B_{1,\om}(v)}\frac{\int_{D(0,\frac12)}v(\z)\dm(\z)}{\om(D(0,\frac12))},
    $$
we have
    \begin{equation*}
    \begin{split}
    (v\om)(\{z:|P^{+}_\om(f1_{\overline{D(0,\frac12)}})(z)|>\lambda\})
    &\le\int_\D\frac{|P^{+}_\om(f1_{\overline{D(0,\frac12)}})(z)|}{\lambda}v(z)\dm(z)\\
    &\le\frac1\lambda\int_\D\left(\int_{D(0,\frac12)}|f(\z)||B^\om_z(\z)|\dm(\z)\right)v(z)\dm(z)\\
    &\lesssim\frac1\lambda\int_{D(0,\frac12)}|f(\z)|\dm(\z)\lesssim\frac{B_{1,\om}(v)\|f\|_{L^1_{\om}(v)}}{\lambda}.
    \end{split}
    \end{equation*}
Moreover,
    \begin{equation*}
    \begin{split}
    (v\om)(\{z:|P^{+}_\om(f)(z)|>\lambda\})
    &\le(v\om)(\{z:|P^{+}_\om(f1_{\overline{D(0,\frac12)}})(z)|+|P^{+}_\om(f1_{R_1})(z)|+|P^{+}_\om(f1_{R_2})(z)|>\lambda\})\\
    &\le(v\om)\bigg(\left\{z:|P^{+}_\om(f1_{\overline{D(0,\frac12)}})(z)|>\frac\lambda3\right\}
    \cup\left\{z:|P^{+}_\om(f1_{R_1})(z)|>\frac\lambda3\right\}\\
    &\qquad\cup\left\{z:|P^{+}_\om(f1_{R_1})(z)|>\frac\lambda3\right\}\bigg)\\
    &\le(v\om)\left(\left\{z:|P^{+}_\om(f1_{\overline{D(0,\frac12)}})(z)|>\frac\lambda3\right\}\right)\\
    &\quad+(v\om)\left(\left\{z:|P^{+}_\om(f1_{R_1})(z)|>\frac\lambda3\right\}\right)\\
    &\quad+(v\om)\left(\left\{z:|P^{+}_\om(f1_{R_2})(z)|>\frac\lambda3\right\}\right),
    \end{split}
    \end{equation*}
so it suffices to show that
    $$
    (v\om)\left(\left\{z:|P^{+}_\om(f1_{R})(z)|>\lambda\right\}\right)\lesssim\frac{\|f1_R\|_{L^1_{\om}(v)}}{\lambda},\quad R\in\{R_1,R_2\},
    $$
for large values of $\lambda$. To see this, fix $R\in\{R_1,R_2\}$,
and decompose $|f1_R|=g+b$ according to Lemma~\ref{Lemma:C-Z} and
the weight $\om\in\R$. Then the definition of $g$ and Lemma~\ref{Lemma:C-Z}(iv) give
    \begin{equation}\label{*}
    |g(z)|\le\sum_k\left(\frac{1_{Q_k}(z)}{\om(Q_k)}\int_{Q_k}|f(\z)|\dm(\z)\right)\lesssim\lambda,\quad z\in\Omega=\cup_kQ_k,
    \end{equation}
which together with Lemma~\ref{Lemma:C-Z}(i) and the definition of $g$ yields
    \begin{equation*}
    \begin{split}
    \|g\|_{L^2_{v\om}}^2&=\int_F|g(\z)|^2v(\z)\dm(\z)+\int_\Omega|g(\z)|^2v(\z)\dm(\z)\\
    &\lesssim\lambda\int_F|f(\z)|v(\z)\dm(\z)+\lambda\int_\Omega|g(\z)|v(\z)\dm(\z).
    \end{split}
    \end{equation*}
Now, since $v\in B_{1,\om}$ and $\om\in\R$ such that $\om([a,b])\asymp\om\left(\left[a,\frac{a+b}2\right]\right)\asymp\om\left(\left[\frac{a+b}2,b\right]\right)$ for all $0\le a,b\le1$ by the hypotheses,
    \begin{equation}
    \begin{split}\label{eq:j1}
    \int_\Omega|g(\z)|v(\z)\dm(\z)
    &=\sum_k\int_{Q_k}|f(\z)|\frac{(v\om)(Q_k)}{\om(Q_k)}\dm(\z)\\
    &\lesssim\sum_k\int_{Q_k}|f(\z)|M_\om(v)(\z)\dm(\z)\\
    &\le B_{1,\om}(v)\sum_k\int_{Q_k}|f(\z)|v(\z)\dm(\z)\\
    &=B_{1,\om}(v) \int_{\Omega}|f(\z)|v(\z)\dm(\z).
    \end{split}
    \end{equation}
Therefore $\|g\|_{L^2_{\om}(v)}^2\lesssim\lambda
B_{1,\om}(v)\|f1_R\|_{L^1_{\om}(v)}$, and thus $g\in L^2_{\om}(v)$.
Since $B_{1,\om}\subset B_{2,\om}$ with $B_{2,\om}(v)\lesssim
B_{1,\om}(v)$, $P^{+}_\om:L^2_{\om}(v)\to L^2_{\om}(v)$ is bounded
by Theorem~\ref{Theorem:BB-generalized}. Consequently,
 \begin{equation*}
    \begin{split}
    (v\om)\left(\left\{z:|P^{+}_\om(g)(z)|>\lambda\right\}\right)
    &=(v\om)\left(\left\{z:|P^{+}_\om(g)(z)|^2>\lambda^2\right\}\right)\\
    &\le\frac1{\lambda^2}\int_\D|P^{+}_\om(g)(\z)|^2v(\z)\dm(\z)
    \lesssim B_{2,\om}(v)\frac{\|g\|_{L^2_{\om}(v)}^2}{\lambda^2}\\
    &\lesssim B^2_{1,\om}(v)\frac{\|f1_R\|_{L^1_{\om}(v)}}{\lambda}.
    \end{split}
    \end{equation*}

To deal with $b$, write $b=\sum_k b_k$, where $b_k=b1_{Q_k}$. Then
$|P^{+}_\om(b)|\le\sum_k|P^{+}_\om(b_k)|$. For each~$k$, let $D_k$
be the circumscribed disc of $Q_k$ with center $z_k$ and let $D_k'$
be the concentric disc of double radius. Further, let
$\Omega'=\cup_k D_k'\cap\D$. Now that $b$ has mean value zero on
$Q_k$,
    \begin{equation}
    \begin{split}
    \int_{Q_k}b(\z)\dm(\z)
    &=\int_{Q_k}\left(f(\z)-g(\z)\right)\dm(\z)
    =\int_{Q_k}f(\z)\dm(\z)\\
    &\qquad-\int_{Q_k}\left(\frac{1}{\om(Q_k)}\int_{Q_k}f(z)\dm(z)\right)\dm(\z)=0,
    \end{split}
    \end{equation}
we deduce
    \begin{equation*}
    \begin{split}
    |P^{+}_\om(b_k)(z)|
    &=\left|\int_{Q_k}b(\z)|\overline{B_z^\omega(\z)}|\dm(\z)-\int_{Q_k}b(\z)|\overline{B_z^\omega(z_k)}|\dm(\z)\right|\\
    &\le\int_{Q_k}|b(\z)||B_z^\omega(\z)-B_z^\omega(z_k)|\dm(\z).
    \end{split}
    \end{equation*}
Consequently,
    \begin{equation*}
    \begin{split}
    &\int_{\D\setminus\Omega'}|P^{+}_\om(b)(z)|v(z)\dm(z)\\
    &\le\sum_k\int_{\D\setminus\Omega'}|P^{+}_\om(b_k)(z)|v(z)\dm(z)\\
    &\le\sum_k\int_{\D\setminus\Omega'}\left(\int_{Q_k}|b(\z)||B_z^\omega(\z)-B_z^\omega(z_k)|\dm(\z)\right)v(z)\dm(z)\\
    &=\sum_k\int_{Q_k}|b(\z)|\left(\int_{\D\setminus\Omega'}|B_z^\omega(\z)-B_z^\omega(z_k)|v(z)\dm(z)\right)\dm(\z)\\
    &\le \sum_k\int_{Q_k}|b(\z)|\left(\int_{\D\setminus D(z_k,2|\z-z_k|)}|B_z^\omega(\z)-B_z^\omega(z_k)|v(z)\dm(z)\right)\dm(\z).
    \end{split}
    \end{equation*}
There is an absolute constant $C>0$ such that $|\z-z_k|\le
C(1-|z_k|)$ for any $k$ and any $\z\in Q_k$. Hence the inner integral in each summand is bounded by a constant times $\inf_{a\in
D(z_k,\sqrt{2}(1-|z_k|)}M_\om(v)(a)$ by
Lemma~\ref{Lemma:local-integrated}. Therefore \eqref{eq:j1} yields
    \begin{equation}
    \begin{split}\label{**}
    \int_{\D\setminus\Omega'}|P^{+}_\om(b)(z)|v(z)\dm(z)
    &\lesssim\sum_k \inf_{a\in D(z_k,\sqrt{2}(1-|z_k|)}M_\om(v)(a)\int_{Q_k}|b(\z)|\dm(\z)\\
    &\le\sum_k\int_{Q_k}|b(\z)| M_\om(v)(\z)\dm(\z)\\
    &\le B_{1,\om}(v)\sum_k\int_{Q_k}|b(\z)|v(\z)\dm(\z)\\
    &\le B_{1,\om}(v)\int_\Omega|b(\z)|v(\z)\dm(\z)\\
    &\le B_{1,\om}(v)\int_\Omega|f(\z)|v(\z)\dm(\z)\\
    &\quad+B_{1,\om}(v)\int_\Omega|g(\z)|v(\z)\dm(\z)\\
    &\lesssim B^2_{1,\om}(v) \int_\Omega|f(\z)|v(\z)\dm(\z)\\
    &\lesssim B^2_{1,\om}(v)\|f1_R\|_{L^1_{\om}(v)}.
    \end{split}
    \end{equation}
Further,
    \begin{equation}\label{***}
    \begin{split}
    (v\om)(\{z:|P^{+}_\om(b)(z)|>\lambda\})
    &\le(v\om)(\{z:|P^{+}_\om(b)(z)|>\lambda\}\cap(\D\setminus\Omega'))\\
    &\quad+(v\om)(\{z:|P^{+}_\om(b)(z)|>\lambda\}\cap\Omega'),
    \end{split}
    \end{equation}
where
    $$
    (v\om)(\{z:|P^{+}_\om(b)(z)|>\lambda\}\cap(\D\setminus\Omega'))
    \le\frac1{\lambda}\int_{\D\setminus\Omega'}|P^{+}_\om(b)(z)|v(z)\dm(z)\lesssim\frac{\|f1_R\|_{L^1_{\om}(v)}}{\lambda}
    $$
by \eqref{**}. Since $\om\in\R$ such that
$\om([a,b])\asymp\om\left(\left[a,\frac{a+b}2\right]\right)\asymp\om\left(\left[\frac{a+b}2,b\right]\right)$
for all $0\le a,b\le1$ by the hypothesis, we have
$\om(Q_k)\asymp\om(D_k'\cap\D)$, and hence (iv) gives
    \begin{equation*}
    \begin{split}
    &(v\om)(\{z:|P^{+}_\om(b)(z)|>\lambda\}\cap\Omega')
    \le(v\om)(\Omega')
    \le\sum_k(v\om)(D_k'\cap\D)\\
    &\le\frac{1}{\lambda} \sum_k
    \frac{(v\om)(D_k'\cap\D))}{\om(Q_k)}\int_{Q_k} |f(z)|\dm(z)\\
    &\lesssim
    \frac{1}{\lambda} \sum_k
    \frac{(v\om)(D_k'\cap\D))}{\om(D_k'\cap\D)}\int_{Q_k}
    |f(z)|\dm(z)\\
    & \le \frac{1}{\lambda} \sum_k \int_{Q_k}|f(z)|M_\om(v)(z)\dm(z)\\
    &\le B_{1,\om}(v) \frac{1}{\lambda}\sum_k \int_{Q_k} |f(z)|v(z)\dm(z)
    \le B_{1,\om}(v)\frac{\|f1_R\|_{L^1_{\om}(v)}}{\lambda}.
    \end{split}
    \end{equation*}
Hence
    \begin{equation}\label{6}
    \begin{split}
    (v\om)(\{z:|P^{+}_\om(f1_R)(z)|>\lambda\})
    &\le(v\om)(\{z:|P^{+}_\om(g)(z)|>\lambda/2\})\\
    &\quad +(v\om)(\{z:|P^{+}_\om(b)(z)|>\lambda/2\})\\
    &\lesssim\frac{\|f1_R\|_{L^1_{\om}(v)}}{\lambda},\quad R\in\{R_1,R_2\},
    \end{split}
    \end{equation}
and thus we get (i). To be precise, this proof works only for $f\in L^1_\om$ because Lemma~\ref{Lemma:C-Z} is applied, but the general case follows by applying \eqref{6} to the function $\min\{f,n\}$ with $f$ non-negative and then letting $n\to\infty$.

Since (i) trivially implies (ii), it remains to show that
(ii) implies (iii). Let $S_1$ and $S_2$ be Carleson squares satisfying the hypothesis in Lemma~\ref{le:poinwisestimate}, and let $f$ a non-negative function supported on $S_1$. Further, choose
    $$
    \lambda=C\frac{\int_{S_1}f(\z)\dm(\z)}{\om(S_1)},
    $$
where $C$ is the constant appearing in \eqref{eq:poinwisestimate}. Since
    $
   \lambda(v\om)(\{z:|P_\om(f)(z)|>\lambda\})\lesssim\|f\|_{L^1_{\om}(v)}
    $
by the hypothesis, it follows by Lemma~\ref{le:poinwisestimate} that there exists $C_1>0$ such that
    $$
    \frac{(f\om)(S_1)}{\om(S_1)}(v\om)(S_2)\le C_1(fv\om)(S_1).
    $$
By choosing $f=1_E\om^{-1}$ for $E\subset S_1$ and applying Lebesgue differentiation theorem, we get
    $$
    \frac{(v\om)(S_2)}{\om(S_1)}\le C_1v(z)
    $$
for almost every $z\in S_1$. Since the same is true when the roles of $S_1$ and $S_2$ are interchanged, we deduce
    $$
    C_1v(z)\ge\frac{(v\om)(S_2)}{\om(S_1)}\ge\frac{\om(S_2)(v\om)(S_1)}{C_1\om(S_2)\om(S_1)}=\frac{1}{C_1}\frac{(v\om)(S_1)}{\om(S_1)}
    $$
for almost every $z\in S_1$. It follows that
    $$
    \sup_{S:z\in S}\frac{(v\om)(S)}{\om(S)}\lesssim v(z)
    $$
for almost every $z\in\D$. This
implies
    \begin{equation}\label{eq:sup1}
    \sup_{z\in D(a,r)}\frac{(v\om)(D(a,r)\cap\D)}{\om(D(a,r)\cap\D)}
    \lesssim v(z)
    \end{equation}
for almost every $z\in\D$, where the supremum runs over the discs touching the boundary. Moreover, the squares $S_1$ and $S_2$ in the statement of Lemma~\ref{le:poinwisestimate} can be replaced by Euclidean discs $D(a_1,
R(1-|a_1|)$ and $D(a_2, R(1-|a_1|)$, where $R$ is fixed and small
enough. By using this fact with the above reasoning in hand and
\eqref{eq:sup1}, we deduce $v\in B_{1,\om}$.
\end{Prf}

\section{Two-weight inequality for the positive operator $P^{+}_{\Psi,\mu}$}\label{Sec5}

The purpose of this section is to prove Theorem~\ref{th:twoweights}.
A reasoning similar to that in the proof of Theorem~\ref{Theorem:BB-generalized} then shows that Theorem~\ref{th:P+logsub} is an immediate consequence of this result.

\begin{theorem}\label{th:twoweights}
Let $1<p<\infty$ and $\mu$ be a positive Borel measure on $\D$, and let $v,u\in L^1_\mu$ non-negative.
Let $\Psi:D(1,1)\to\mathbb{C}$ be an analytic function such that its restriction to the interval $(0,2)$ is positive and the following conditions
hold:
\begin{enumerate}
\item[\rm(i)] $|\Psi(1-z)|\asymp \Psi(|1-z|) $ for all $z\in \D$;
\item[\rm(ii)] $\Psi$ is essentially decreasing on $(0,2)$;
\item[\rm(iii)] There exists a constant $C>0$ such that $\Psi(t)\le C \Psi(2t)$ for all $t\in(0,1)$;
\item[\rm(iv)] $|\Psi(1-z)|=|\Psi(1-\overline{z})| $ for all $z\in\D$.
\end{enumerate}
Then $P^{+}_{\Psi,\mu}: L^p_\mu(v)\to L^p_\mu(u)$ is bounded if and only
if there exist constants $C_0=C_0(p,\mu,v,u)>0$ and $C^\star_0=C^\star_0(p,\mu,v,u)>0$ such that
    \begin{equation}\label{j1}
     \left\|\M_{u^{1/p}}P^{+}_{\Psi,\mu}\M_{\sigma^{1/{p'}}}(1_{S}\sigma^{1/p})\right\|_{L^p_\mu}\le C_0
    \left\|1_{S}\sigma^{1/p}\right\|_{L^p_\mu}
\end{equation}
and
    \begin{equation}\label{j2}
    \left\|\M_{\sigma^{1/{p'}}}P^{+}_{\Psi,\mu}\M_{u^{1/p}}(1_{S}u^{1/p'})\right\|_{L^{p'}_\mu}\le C^\star_0
    \left\|1_{S}u^{1/p'}\right\|_{L^{p'}_\mu}
    \end{equation}
for all Carleson squares $S\subset\D$, where $\sigma=v^{1-p'}$. Moreover, there exists a constant $C_1=C_1(p,\mu)>0$  such that
    $$
    \left\| P^{+}_{\Psi,\mu}\right\|_{L^p_\mu(v)\to L^p_{\mu}(u)}\le C_1(C_0+C_0^\star).
    $$
\end{theorem}

As in the one-weight case $P^+_{\Psi,\mu}:L_\mu^p(v)\to L_\mu^p(v)$ given in Corollary~\ref{co:sufp>1} it is more convenient to consider first a dyadic model. To do this, let $\mathbb E^{\mu}_{S}f$ and $\mathbb E^{\sigma\mu}_{S}f$ denote the expectations of a function $f$ over a square $S$ with respect to the measures $\mu$ and $\sigma\,d\mu$, respectively.  Given a dyadic grid $\mathcal D$ on $\T$ and a sequence $\tau=\{\tau_{S(I)}\}_{I\in \mathcal D}$ of nonnegative numbers, consider the dyadic positive operator defined by
    \begin{equation}\label{e.dyadop0}
    T(f)=T_{\mu,\tau,\mathcal{D}}(f)=\sum_{I\in\mathcal D}\tau_{S(I)}(\mathbb E^{\mu}_{S(I)}f)1_{S(I)}.
    \end{equation}
   Given $I\subset\mathcal{ D}$  we can identify it with its associated Carleson
square $S(I)$. So, via this identification, for a dyadic grid $\mathcal D$ on $\T$ we shall simply write
 \begin{equation}\label{e.dyadop}
    T(f)=T_{\mu,\tau,\mathcal{D}}(f)=\sum_{I\in\mathcal D}\tau_{S}(\mathbb E^{\mu}_{S}f)1_{S}
    \end{equation}
for the corresponding dyadic positive operator.

The following theorem characterizes the boundedness of the operator $T$ in the two-weight setting.  See \cite{PottRegueraJFA13, Sawyer2, Treil12}.

\begin{theorem}\label{th:maind}
Let $1<p<\infty$, $\mu$ be a positive Borel measure on $\D$, $\sigma,u\in L^1_\mu$ non-negative and let $T=T_{\mu,\tau,\mathcal{D}}$ be the dyadic positive operator defined in \eqref{e.dyadop}. Then $T(\sigma\cdot):L_\mu^{p}(\sigma)\to L_\mu^{p}(u)$ is bounded if and only if there exist constants $C_0=C_0(p,\mu,\s,u)>0$ and $C^\star_0=C^\star_0(p,\mu,\s,u)>0$ such that
    \begin{equation}\label{e.test}
    \| T(\sigma 1_{S})\|_{L_\mu^{p}(u)}^{p}\le C_0(\sigma\mu)(S)
    \end{equation}
and
    \begin{equation}\label{e.testd}
    \| T^\star(u 1_{S})\|_{L_\mu^{p'}(\sigma)}^{p'}\le C^\star_0(u\mu)(S)
    \end{equation}
for all $S\in\mathcal D$. Moreover, there exists a constant $C_1=C_1(p,\mu)>0$ such that
    $$
    \|T(\sigma\cdot)\|_{L_\mu^{p}(\sigma)\to  L_\mu^{p}(u)}\le C_1(C_0 + C^\star_0).
    $$
\end{theorem}

Let now $\sigma$ be a weight and $f$ a locally integrable function in $\D$. Let $S_{0}\in\mathcal D$ and denote $\mathcal D_{0}=\{S\in\mathcal D:S\subset S_{0}\}$. Further, let
    $$
    \mathcal L(S_{0})=\{S\in\mathcal D_{0}:S \text{ is a maximal Carleson square in $\mathcal{D}_0$ such that } \mathbb E^{\sigma\mu}_{S}|f|
    > 4\mathbb E^{\sigma\mu}_{S_{0}}|f|\}.
    $$
Define $\mathcal L_0= \{S_0\}$ and $\mathcal L_{i}= \cup_{L\in\mathcal L_{i-1} }\mathcal L(L)$ for all $i\in\N$, and denote the union of all the stopping squares by $\mathcal L= \cup_{i\geq 0} \mathcal L_{i}$. For $S\in\mathcal D_0$, let $\lambda(S)$ be the minimal square $L\in \mathcal L$ such that $S\subset L$ and let $\mathcal D(L)=\{S\in\mathcal D_{0}:\lambda(S)=L\}$.

The stopping squares $\mathcal{L}$ can be used to linearise the maximal function $M_\nu$. More precisely, we have the pointwise estimate
    \begin{equation}\label{e.pointwise}
    \sum_{L\in\mathcal L}(\mathbb E_{L}^{\sigma\mu}|f|)1_{L}(z) \lesssim M_{\sigma\mu}f(z), \quad z\in \D.
    \end{equation}
To see this, assume $z\in S_{0}$ for some $S_{0}\in\mathcal L_{0}$, for otherwise the inequality is trivial because the left hand side is zero.
Then there exists a stopping square $L'\in \LL$ with minimal side length containing $z$. The expectations increase geometrically, that is,
    $$
    \mathbb E_{L}^{\sigma\mu}|f|>4\mathbb E_{\widetilde{L}}^{\sigma\mu}|f|, \quad L,\widetilde{L}\in \mathcal L, \quad L\subsetneq \widetilde{L},
    $$
therefore
    $$
    \sum_{\substack{L\in \mathcal L\\ z\in L}} \mathbb {E}_{L}^{\sigma\mu}|f|\le \mathbb {E}_{L'}^{\sigma\mu}|f| \sum_{j=0}^\infty 4^{-j} \lesssim M_{\sigma\mu}f(z),
    $$
concluding the proof of \eqref{e.pointwise}.

An application of \eqref{e.pointwise} and \eqref{e.mbounded}
provides the useful inequality
    \begin{equation}
    \label{e.lbounded}
    \sum_{L\in \mathcal L} (\mathbb E_{L}^{\sigma\mu}|f|)^{p} \sigma\mu(L)\lesssim \| f\|_{L_\mu^{p}(\sigma)}.
    \end{equation}

\begin{Prf}{\em{Theorem~\ref{th:maind}}.}
We will assume there is a finite collection of dyadic squares $\mathcal S$ in the definition of the operator $T$, and we will prove the operator norm is independent of the chosen collection. So from now on
    $$
    Tf=\sum_{S\in \mathcal S} \tau_{S}(\mathbb E^{\mu}_{S}f) 1_{S}.
    $$
It is enough to prove boundedness of the bilinear form $\langle T(\sigma f), gu \rangle_{L^2_\mu}$, where $f\in L_\mu^{p}(\sigma)$ and $g\in L_\mu^{p'}(u)$ are positive. Following the argument in \cite{Treil12}, we seek an estimate of the form
    \begin{equation}\label{e.goal}
    \langle T(\sigma f),gu\rangle_{L^2_\mu}\le A\| f\|_{L_\mu^{p}(\sigma)}\| g\|_{L_\mu^{p'}(u)}+B\| f\|_{L_\mu^{p}(\sigma)}^{p}.
    \end{equation}
We first divide the squares in $\mathcal S$ into two collections $\mathcal S_{1}$ and $\mathcal S_{2}$ according to the following criterion. A square $S$ will belong to $\mathcal{S}_{1}$, if
    \begin{equation}\label{e.crit}
    \left( \mathbb E^{\mu \sigma}_{S}f\right)^{p}\mu \sigma(S)\geq \left( \mathbb E^{\mu u}_{S}g\right)^{p'}\mu u(S),
    \end{equation}
and it will belong to $\mathcal{S}_{2}$ otherwise.
This reorganization of the Carleson squares allows us to write $T=T_{1}+T_{2}$, where
    $$
    T_{i}f=\sum_{S\in \mathcal S_{i}}\tau_{S}(\mathbb E^{\mu}_{S}f )1_{S},\quad i=1,2.
    $$
The idea of writing $T$ as the sum of $T_1$ and $T_2$ was already present in the work of Treil \cite{Treil12} and previously in the work of Nazarov, Treil and Volberg \cite{NaTrVoAmerJ99}. We will prove boundedness of $T_{1}$ using the testing condition \eqref{e.test}. The boundedness of $T_{2}$ can be proven analogously to $T_1$, only using \eqref{e.testd} this time. First note that
    \begin{equation*}
    \begin{split}
    \langle T_{1}(\sigma f), gu \rangle_{L^2_\mu}
    &=\sum_{S\in \mathcal S_{1}} \tau_{S}\mathbb E^{\mu}_{S}(f\sigma) \langle gu, 1_{S}\rangle_{L^2_\mu}
    =\sum_{L\in \mathcal L}\sum_{S\in \mathcal D(L)}  \tau_{S}\mathbb E^{\mu}_{S}(f\sigma) \langle gu, 1_{S}\rangle_{L^2_\mu}\\
    &=\sum_{L\in \mathcal L} \langle T_{L}(\sigma f), gu \rangle_{L^2_\mu},
    \end{split}
    \end{equation*}
where $\LL$ is a collection of stopping Carleson squares in the family $\mathcal{S}_1$, to be specified below, and
    $
    T_{L}f= \sum_{S\in \mathcal D(L)}\tau_{S}\mathbb E^{\mu}_{S}(f)1_{S}.
    $
To find the collection of stopping Carleson squares $\LL$, we define $\mathcal L_{0}$ as the collection of maximal Carleson squares in the family $\mathcal S_{1}$, and follow the definition after Theorem~\ref{th:maind} for given $f$ and $\sigma$ to define $\mathcal L$, with $\mathcal S_1 $ as our family of dyadic Carleson squares. Clearly,
    \begin{equation}
    \label{e.nose}
    \sum_{L\in\mathcal L} \langle T_{L}(\sigma f), gu \rangle_{L^2_\mu}
    =\sum_{L\in\mathcal L}\int T_{L}(\sigma f)(z) g(z)u(z)\,d\mu(z)= I+II,
    \end{equation}
where
    $$
    I=\sum_{i}\sum_{L\in\mathcal L_{i}} \int_{\displaystyle L\setminus \cup_{\substack{L'\in \mathcal L_{i+1}\\L'\subset L}}L'}T_{L}(\sigma f)(z) g(z)u(z)\,d\mu(z),
    $$
and
    $$
    II=\sum_{i}\sum_{L\in\mathcal L_{i}} \int_{\displaystyle \cup_{\substack{L'\in \mathcal L_{i+1}\\L'\subset L}}L'}T_{L}(\sigma f)(z) g(z)u(z)\,d\mu(z).
    $$
To deal with $I$, we estimate the norm of $T_{L}$. By using the testing condition \eqref{e.test} and the fact that $S\in \mathcal D(L)$ are not stopping Carleson squares, we deduce
    \begin{equation}\label{e.normtl}
    \begin{split}
    \|T_{L}(\sigma f)\|_{L_\mu^{p}(u)}^{p}
    &=\left\|  \sum_{S\in \mathcal D(L)}  \tau_{S}\mathbb E^{\mu}_{S}(f\sigma)1_{S}\right\|_{ L_\mu^{p}(u)}^{p}
    =\left\|  \sum_{S\in \mathcal D(L)} \frac{\mu \sigma(S)}{\mu (S)}  \tau_{S}\mathbb E^{\sigma\mu}_{S}(f)1_{S}\right\|_{ L_\mu^{p}(u)}^{p}\\
    &\le4^p\left(\mathbb E_{L}^{\sigma\mu}(f)\right)^{p} \left\|  \sum_{S\in \mathcal D(L)} \frac{\mu \sigma(S)}{\mu (S)} \tau_{S}1_{S}\right\|_{ L_\mu^{p}(u)}^{p}\\
    &\le4^p\left(\mathbb E_{L}^{\sigma\mu}(f)\right)^{p} \left\|  T(\sigma 1_{L})\right\|_{ L_\mu^{p}(u)}^{p}
    \le4^pC_{0}  \left(\mathbb E_{L}^{\sigma\mu}(f)\right)^{p}\sigma\mu(L).
    \end{split}
    \end{equation}
Since $\displaystyle \cup_{i}\cup_{L\in \mathcal L_{i}}L\setminus
\cup_{\substack{L'\in \mathcal L_{i+1}\\L'\subset L}} L'$ forms a
collection of disjoint sets in $\mathcal L_{0}$, H{\"o}lder's
inequality,  \eqref{e.normtl} and \eqref{e.lbounded} yield
    \begin{equation}\label{1}
    \begin{split}
    I&\le\sum_{i}\sum_{L\in\mathcal L_{i}} \| T_{L}(f\sigma)\|_{ L_\mu^{p}(u)} \| g1_{L\setminus \cup_{\substack{L'\in \mathcal L_{i+1}\\L'\subset L}}L'}\|_{L_\mu^{p'}(u)}\\
    &\le\left( \sum_{i}\sum_{L\in\mathcal L_{i}} \| T_{L}(f\sigma)\|_{ L_\mu^{p}(u)}^{p}  \right)^{1/p} \left( \sum_{i}\sum_{L\in\mathcal L_{i}} \| g1_{L\setminus \cup_{\substack{L'\in \mathcal L_{i+1}\\L'\subset L}}L'}\|_{L_\mu^{p'}(u)}^{p'}\right)^{1/p'}\\
    &\le4 C_{0}^{1/p}\left ( \sum_{L\in \mathcal L} \left(\mathbb E_{L}^{\mu \sigma}f\right)^{p} \sigma\mu(L)\right)^{1/p}\|g\|_{L_\mu^{p'}(u)}
    \lesssim C_{0}\| f\|_{L_\mu^{p}(\sigma)} \| g\|_{L_\mu^{p'}(u)}.
    \end{split}
    \end{equation}
We now turn to II. If $L\in \mathcal L$ be fixed, then the operator
$T_{L}(f\sigma)$ is constant on $L'$, where $L'\in \mathcal L, \,
L'\subsetneq L$. That is, $L'$ is contained in some Carleson square
$S$ of the family $\mathcal{D}(L)$.  We will denote this constant by
$T_{L}(f\sigma)(L')$. For a fixed $L\in \mathcal L_{i}$, this,
H{\"o}lder's inequality, \eqref{e.normtl} and the hypothesis
\eqref{e.crit} yield
    \begin{equation*}
    \begin{split}
    \int_{ \cup_{\substack{L'\in \mathcal L_{i+1}\\L'\subset L}}L'}T_{L}(\sigma f)(z) g(z)u(z)\,d\mu(z)
    &=\sum_{\substack{L'\in \mathcal L_{i+1}\\L'\subset L}} T_{L}(f\sigma)(L')\int_{L'}g(z)u(z)\,d\mu(z)\\
    &=\sum_{\substack{L'\in \mathcal L_{i+1}\\L'\subset L}} T_{L}(f\sigma)(L')\frac{\int_{L'}g(z)u(z)\,d\mu(z)}{\mu u(L')}\mu u(L')\\
    &=\sum_{\substack{L'\in \mathcal L_{i+1}\\L'\subset L}} \int_{L'} T_{L}(f\sigma)(z) \left(\mathbb E_{L'}^{\mu u}g\right)u(z)\,d\mu(z)\\
    &=\int_{L} T_{L}(f\sigma)(z) \left(\sum_{\substack{L'\in \mathcal L_{i+1}\\L'\subset L}} \mathbb E_{L'}^{\mu u}g 1_{L'}(z)\right) u(z)\,d\mu(z)\\
    &\le\left\|   T_{L}(f\sigma)\right\|_{L_\mu^{p}(u)  }    \left\|      \sum_{\substack{L'\in \mathcal L_{i+1}\\ L'\subset L}} \mathbb E_{L'}^{\mu u}g 1_{L'} \right\|_{L_\mu^{p'}(u)}\\
    &=\| T_{L}(f\sigma)\|_{L_\mu^{p}(u)} \left( \sum_{\substack{L'\in \mathcal L_{i+1}\\L'\subset L}} (\mathbb E_{L'}^{\mu u}g)^{p'}\mu u(L')  \right)^{1/p'}\\
    &\le4 C_0 (\mathbb E_{L}^{\mu \sigma}f) \mu \sigma(L)^{1/p}\left( \sum_{\substack{L'\in \mathcal L_{i+1}\\L'\subset L}} (\mathbb E_{L'}^{\mu \sigma}f)^{p}\mu \sigma(L') \right)^{1/p'}.
    \end{split}
    \end{equation*}
By summing this estimate in $L$ and using \eqref{e.lbounded},  we obtain
    \begin{equation}\label{2}
    \begin{split}
    II&\lesssim\sum_{i}\sum_{L\in\mathcal L_{i}} (\mathbb E_{L}^{\mu \sigma}f) \mu \sigma(L)^{1/p}\left( \sum_{\substack{L'\in \mathcal L_{i+1}\\L'\subset L}} (\mathbb E_{L'}^{\mu \sigma}f)^{p}\mu \sigma(L') \right)^{1/p'}\\
    &\lesssim \left( \sum_{L\in \mathcal L} (\mathbb E_{L}^{\mu \sigma} f)^{p}\mu \sigma(L)\right)^{1/p} \left( \sum_{i}\sum_{L\in\mathcal L_{i}}\sum_{\substack{L'\in \mathcal L_{i+1}\\L'\subset L}} (\mathbb E_{L'}^{\mu \sigma}f)^{p}\mu \sigma(L') \right)^{1/p'}\\
    &\lesssim\| f\|_{L_\mu^{p}( \sigma)} \| f\|_{L_\mu^{p}(\sigma)}^{p/p'} \lesssim \| f\|_{L_\mu^{p}(\sigma)}^{p}.
    \end{split}
    \end{equation}
By combining \eqref{1} and \eqref{2}, we get \eqref{e.goal}.
\end{Prf}

\medskip

We now turn to the two-weight inequality for the case of the operator $P^{+}_{\Psi,\mu}$ and its associated dyadic model $P^{\beta}_{\Psi,\mu}$.

Taking $\mathcal{D}^{\beta}$, $\beta\in\{0,1/2\}$,
one of the dyadic grids on $\T$ defined in \eqref{concretegrid} and choosing $\tau_{S(I)}=\frac{\Psi(|I|)\mu(S_{I})}{|I|}$, we obtain the following result as a byproduct of Theorem~\ref{th:maind}.

\begin{corollary}\label{c.twodB}
Let $\beta\in\{0,1/2\}$, $\Psi$ be a positive function on $(0,2)$, $\mu$ be a positive Borel measure on $\D$ and $\sigma,u\in L^1_\mu$ non-negative. Then
    $
    P_{\Psi,\mu}^{\beta}(\sigma\cdot):L_\mu^{p}(\sigma)\rightarrow  L_\mu^{p}(u)
    $
is bounded if and only if there exist constants $C_0=C_0(p,\mu,\sigma,u)>0$ and $C^{\star}_0=C^\star_0(p,\mu,\sigma,u)>0$ such that
    \begin{equation}\label{e.testpbeta}
    \left \| \sum_{\substack{I\in \mathcal D^{\beta}\\ I\subset I_{0}}}\left\langle \sigma 1_{S(I_0)}, \frac{\Psi(|I|)1_{S(I)}}{|I|}\right\rangle_{L^2_\mu} 1_{S(I)} \right\|_{L_\mu^{p}(u)}^{p}\le C_0\mu\sigma(S(I_0)),
    \end{equation}
and
    \begin{equation}\label{e.testpbetad}
    \left\| \sum_{\substack{I\in \mathcal D^{\beta}\\ I\subset I_{0}}}\left \langle u 1_{S(I_0)}, \frac{\Psi(|I|)1_{S(I)}}{|I|}\right\rangle_{L^2_\mu} 1_{S(I)}\right\|_{L^{p'}_\mu(\sigma)}^{p'}\le C^\star_0 \mu u(S(I_0)),
    \end{equation}
for all $I_0\in\mathcal D^{\beta}$.
Moreover, there exists a constant $C_1>0$ independent of the weights, such that
    $$
    \left\|P_{\Psi,\mu}^{\beta}(\sigma\cdot) \right\|_{L^p_\mu(\sigma) \to L_\mu^p(u)} \le C_1(C_0+C^\star_0).
    $$
\end{corollary}

\medskip

\begin{Prf}{\em{ Theorem~\ref{th:twoweights}}.}
By the equivalence of (A) and (C), $P^{+}_{\Psi,\mu}: L^p_\mu(v)\to L^p_\mu(u)$ is bounded if and only if
$\M_{u^{1/p}}P^+_{\Psi,\mu}\M_{\sigma^{1/p'}}:L^p_\mu\to L^p_\mu$ is bounded.
By the hypothesis (iv), the adjoint of
$\M_{u^{1/p}}P^+_{\Psi,\mu}\M_{\sigma^{1/p'}}$ with respect to the
$L^2_\mu$-pairing is $\M_{\sigma^{1/p'}}P^+_{\Psi,\mu}\M_{u^{1/p}}$.
Consequently, the necessity of the conditions \eqref{j1} and \eqref{j2} is obvious.
Conversely, by the first inequality in \eqref{eq:kernelcomp}, the
testing conditions \eqref{j1} and \eqref{j2} imply the corresponding
testing conditions for each $P^\beta_{\Psi,\mu}$,
$\beta\in\{0,1/2\}$, that is, conditions \eqref{e.testpbeta} and
\eqref{e.testpbetad}, and therefore the boundedness of each operator
$P_{\Psi,\mu}^{\beta}(\sigma\cdot):L_\mu^{p}(\sigma)\rightarrow  L_\mu^{p}(u)$,
$\b\in\{0,1/2\}$, by Corollary~\ref{c.twodB}. The second inequality in \eqref{eq:kernelcomp} now implies the boundedness of
$P_{\Psi,\mu}^{+}:L_\mu^{p}(\sigma)\rightarrow  L_\mu^{p}(u)$ with the required norm bounds. Finally, by using the equivalence of (A) and (B) and \eqref{equivnorms}, we deduce that $P_{\Psi,\mu}^{+}:L_\mu^{p}(v)\rightarrow  L_\mu^{p}(u)$ is bounded with the claimed norm bound.
\end{Prf}

 \end{document}